\documentclass[10pt]{article}

\usepackage{amsfonts}
\usepackage{amssymb}
\usepackage{amsthm}
\usepackage{amsmath}
\usepackage{bm}
\usepackage{amscd} 
\usepackage{mathrsfs}

\usepackage[all]{xy}
\usepackage{graphicx}
\usepackage{authblk}
\usepackage{hyperref}

\numberwithin{equation}{section} \DeclareMathSizes{2}{10}{12}{13}
\newtheorem{thm}{Proposition}[section]
\newtheorem{Thm}[thm]{Theorem}
\newtheorem{rem}[thm]{Remark}

\newtheorem{cor}[thm]{Corollary}
\newtheorem{lem}[thm]{Lemma}
\newtheorem{defn}[thm]{Definition}
\newtheorem{ex}[thm]{Examples}

\title{Closure operations, Continuous valuations on monoids and Spectral spaces}

\author{Samarpita Ray\footnote{Department of Mathematics, Indian Institute of Science, Bangalore, India.\\ email: ray.samarpita31@gmail.com.
}}
\date{}
\begin{document}
\maketitle

\begin{abstract}
We present several naturally occurring classes of spectral spaces using commutative algebra on pointed monoids. For this purpose, our main tools are finite type closure operations and continuous valuations on monoids which we introduce in this work.  In the process, we make a detailed study of different closure operations on monoids. We prove that the collection of continuous valuations on a topological monoid with topology determined by any finitely generated ideal is a spectral space.
\end{abstract}

\section{Introduction}
In basic commutative algebra, topological spaces arise essentially in two contexts. First, as the Zariski spectrum  $Spec(A)$ of a commutative ring $A$ and secondly while topologizing a ring with the $I$-adic topology, where $I$ is an ideal of the ring (see, for 
instance, \cite{AM}). A celebrated result by Hochster (\cite{H}) shows that  a topological space $X$ is homeomorphic to $Spec(A)$ for some commutative ring $A$ if and
only if $X$ has the following properties :\smallskip
\newline(1) $X$ is $T_0$,\\
(2) $X$ is quasi-compact (i.e., any open cover of $X$ admits a finite
subcover), \\
(3) $X$ admits a basis of quasi-compact open subspaces that is closed under
finite intersections\\
(4) and every irreducible closed subspace $Y$ of $X$ has a unique
generic point.\\ The striking thing about the above description is the fact that it is purely topological
in nature. 
Such a topological space $X$ is called {\it spectral}. Unlike the Zariski spectrum, the relation between the $I$-adic topology and spectral spaces is not so apparent. However, one of the most important tools to understand the $I$-adic topology is through valuations (see, for instance, \cite{B}) and several spectral spaces naturally emerge on studying valuations. For example, given any field $F$, the set of all valuation rings
having the quotient field $F$ is a spectral space (see, for instance, \cite[Example 2.2 (8)]{O}). The valuation spectrum of any ring is a spectral space \cite{Hu}. Moreover, for any ring $A$ with the $I$-adic topology having a finitely generated ideal of definition, the space of continuous valuations $Cont(A)$ forms a spectral space \cite{Hu}. On the one hand, the Zariski spectra are the building blocks of schemes in algebraic geometry and on the other, the space $Cont(A)$ is the starting point of non-archimedean geometry. This shows that spectral spaces help in linking algebra, topology and geometry.\smallskip
\newline Recently, Finocchiaro \cite{F} developed a new criterion involving {\it ultrafilters} to characterize spectral spaces. This helped them to show that certain familiar objects appearing in commutative algebra can be realized as the spectrum of a ring (see, for instance, \cite{Sp}, \cite{Fo}). For example, Finocchiaro, Fontana and Spirito proved that the collection of submodules of a module over a ring has the structure of a spectral space \cite{Fo}. They further used closure operations from classical commutative ring theory to bring more such spectral spaces in light. Moreover, the methods of \cite{Fo} were used in \cite{AB1}, \cite{AB2} to uncover many spectral spaces coming from objects in an abelian
category and from modules over tensor triangulated categories. In this paper, we use both Hochster's characterization and Finocchiaro's criterion to present natural classes of spectral spaces involving pointed monoids.\smallskip
\newline Our interest in monoids began by looking at the paper \cite{FW} by Weibel and Flores which studies certain geometric structures involving monoids. This interest in the geometry over monoids lies in its natural association to toric geometry, which was pointed out in the work of Cortiñas, Haesemeyer, Walker and Weibel \cite{CHWW}. A detailed work on the commutative and homological algebra on monoids was presented by Flores in \cite{JF}. As such, in this work we look into the topological aspects of monoids via spectral spaces. We begin by showing that the collection of all prime ideals of a monoid or, in other words, the spectrum of a monoid, endowed with the Zariski topology is homeomorphic to the spectrum of a ring, i.e., it is a spectral space. We further prove that the collection of all ideals as well as the collection of all proper ideals of a monoid are also spectral spaces (Corollary  \ref{allideal}). As in \cite{JF}, the notion of $A$-sets over a monoid $A$ is the analogue of the notion of modules over a ring. We introduce closure operations on monoids and obtain natural classes of spectral spaces using finite type closure operations on $A$-sets. In the process, different notions of closure operations like integral, saturation, Frobenius and tight closures are introduced for monoids inspired by the corresponding closure operations on rings from classical commutative algebra. We discuss their persistence and localization properties in detail. Other than rings, valuation has also been studied on ring-like objects like semirings (see, for instance \cite{JJ}, \cite{NP}). Valuation monoids were studied in \cite{CHWW}. In this work, we prove that the collection of all valuation monoids having the same group completion forms a spectral space (Proposition \ref{vmsp}) and that the valuation spectrum of any monoid gives a spectral space (Theorem \ref{4.23}). Finally, we prove that the collection of continuous valuations on a topological monoid whose topology is determined by any finitely generated ideal (in the sense of Definition \ref{I-top}) also gives a spectral space (Theorem \ref{Contval}).
\smallskip
\newline The paper is organized as follows. In $\S$2, we recall some background materials on commutative algebra on monoids and spectral spaces. In $\S$3, we obtain classes of spectral spaces using $A$-sets over a monoid $A$ and in particular using ideals of a monoid $A$. For instance, we prove that the collection of $A$-subsets of an $A$-set forms a spectral space in a manner similar to \cite{Fo}. Following this, in $\S$4, we introduce different closure operations on monoids and study their persistence and localization properties as well. This is done with a view towards obtaining spectral spaces involving finite type closure operations. Finally, in $\S$5, we study valuations and continuous valuations on monoids and several classes of spectral spaces appear.  \\
\\
{\bf Acknowledgements:} I am grateful to my Ph.D advisor Professor Abhishek Banerjee for many helpful discussions on this work. I would like to thank Professor Sudesh Kaur Khanduja and Professor Pooja Singla for their helpful comments and suggestions. I would also like to express my gratitude to the anonymous referee for a very careful reading of the paper and for several useful comments and suggestions.

\section{Background}
\subsection{Monoids}
Throughout this work, by a monoid we shall always mean a commutative pointed (unique basepoint 0) monoid with identity element 1. A morphism of monoids is a basepoint preserving map defined in the obvious way. Let us now briefly recall the commutative algebra on monoids from \cite{JF}.\smallskip
\newline 
A monoid $A$ is called {\it cancellative} if for all $a,b,c \in A$ with $a \neq 0$, $ ab = ac$ implies $b=c$.
 An {\it ideal} $I$ of a monoid $A$ is a pointed subset such that $IA \subseteq A$. An ideal $I \subseteq A$ is generated by a subset $Y$ of $I$ if every $x \in I$ can be written as $ay$ for some $a \in A$ and some $y \in Y$.  If this $Y$ can be chosen to be finite, $I$ is called a {\it finitely generated} ideal.
 An ideal $\mathfrak{p} \neq A$ is called {\it prime} if the complement $A\setminus \mathfrak{p}$ is closed under multiplication, or equivalently if $A/\mathfrak{p}$ is torsion free (see, for instance, \cite[$\S$ 2.1.3]{JF}).
The complement of the set of units $A^{\times}$ of a monoid $A$ forms an ideal and is therefore the unique {\it maximal} ideal of $A$, written as $m_A$. {\it  In this sense, every monoid is local}. A morphism $f: A \longrightarrow B$ is {\it local} if $f(m_A) \subseteq m_B$.\smallskip
\newline Let $I \subseteq A$ be an ideal. The subset $I \times \{0\} \subseteq A \times A$ generates a congruence, i.e., an equivalence relation compatible with the monoid operation, whose associated quotient monoid is written as $A/I$ and identifies all elements of $I$ with
$0$ leaving $A\setminus I$ untouched. 
\newline An $A$-set is a commutative pointed set $M$ with a binary operation $\cdot: A \times M \longrightarrow M$ satisfying, $(i)$ $1.x = x$, $(ii)$ $0_A.x=0_X\ ;\ a.0_X = 0_A$ and $(iii)$ $(a.b).x = a.(b.x)$ for every $a,b \in A$ and $x \in M$, where $0_A$ denotes the basepoint of $A$ and $0_M$ denotes the basepoint of $M$. Morphisms of $A$-sets are defined in the obvious way.
 A subset $N$ of an $A$-set $M$ is called an $A$-{\it subset} of $M$ if $ay \in N$ for every $a \in A$ and $y \in N$.
An $A$-set $M$ is said to be {\it finitely generated} if there exists a finite subset $S\subseteq M$ such that every $x \in M$ can be written as $x = az$ for some $a \in A$ and $z \in S$. An $A$-set is {\it Noetherian} if every $A$-subset is finitely generated.\smallskip
\newline Localization of a monoid $A$ or of an $A$-set are defined in the obvious way. When $S$ is the multiplicatively closed subset of non-zero divisors of $A$, the localization  $S^{-1}A$ is called the {\it total monoid of fractions}. When $A$ is torsion free, i.e., 0 generates a prime ideal, the localization $A^+:= A_{(0)}$ of $A$ at $(0)$ is called the {\it group completion} of $A$. Note that it is $A^+ \setminus 0$ which is a group but not $A^+$. The canonical map $A \rightarrow A^{+}$ is an inclusion only when $A$ is cancellative.


\subsection{Spectral spaces}
A topological space $X$ is called a spectral space (after Hochster \cite{H}) if it is homeomorphic to the spectrum $Spec (R)$ of a ring $R$ with the Zariski topology. As stated in the introduction, spectral spaces can be described in purely topological terms. A map $f:X \longrightarrow Y$ of spectral spaces is called
a {\it spectral map} if for any open and quasi-compact subspace $Y'$ of $Y$, the set $f^{-1}(Y')$ is open and quasi-compact. In particular, any spectral map of spectral
spaces is continuous.\smallskip
\newline A spectral space $X$ need not be Hausdorff unless $X$ is zero dimensional. However, there is a natural way to refine the topology of
$X$ in order to make $X$ an Hausdorff space without losing compactess. The {\it patch topology} (or {\it constructible topology}) on $X$ is the topology which has as a sub-basis for its closed sets the closed sets and quasi-compact open sets of the original space, or in other words, which has the quasi-compact open sets and their complements as an open sub-basis. Endowing $X$ with the patch topology makes it a Hausdorff spectral space (see, \cite[$\S$ 2]{H}) and it is denoted as $X_{patch}$. By a {\it patch} in $X$ we mean a subset of $X$ closed in $X_{patch}$. If $Y$ is a patch in a spectral space $X$, then $Y$ is also a spectral space with the subspace topology (see, \cite[Proposition 9]{H} or \cite[Proposition 2.1]{O} ). We shall use this property in Section \ref{sec 4} to show certain topological spaces are spectral.
\smallskip
\newline Another way of understanding spectral spaces is via {\it ultrafilters}. Let us recall the definition of ultrafilters (say from \cite{F}) before we state the criterion. 
\begin{defn}\label{filter}
A nonempty collection $\mathcal{F}$ of subsets
of a given set $X$ is called a filter on $X$ if the following properties hold:\\
$(i)$ $\emptyset \notin \mathcal{F}$.\\
$(ii)$ If $Y,Z \in \mathcal{F}$, then $Y\cap Z \in \mathcal{F}$.\\
$(iii)$ If $Z \in \mathcal{F}$ and $Z\subset Y \subset X$, then $Y \in \mathcal{F}$.\\
A filter $\mathcal{F}$ on $X$ is an ultrafilter if and only if it is maximal. Equivalently, a filter $\mathcal{F}$ is an ultrafilter on $X$ if for each subset $Y$ of $X$, either $Y \in \mathcal{F}$ or $X \setminus Y \in \mathcal{F}$. We shall denote an ultrafilter by $\mathcal{U}$.
\end{defn}
For further details and examples of filters, see, for instance, \cite{Je}. 
\begin{Thm}\label{2.2}
(see,\cite[Corollary 3.3]{F})
Let $X$ be a topological space. \\
(1) The following conditions are equivalent:\smallskip
\newline $(i)$ $X$ is a spectral space.\\
$(ii)$ $X$ satisfies the $T_0$-axiom and there is a subbasis $\mathbb{S}$ of $X$ such that 
\begin{equation*}
X_{\mathbb{S}}({\mathcal{U}}):= \{x \in X\ |\ [\forall S\in \mathbb{S}, x \in S \iff S \in \mathcal{U}]\} \neq \emptyset
\end{equation*}
for any ultrafilter $\mathcal{U}$ on $X$.\\
(2) If the previous equivalent conditions hold and $\mathbb{S}$ is as in $(ii)$, then a subset
$Y$ of $X$ is closed, with respect to the patch topology, if and only if
$Y_{\mathbb{S}}({\mathcal{V}}):= \{x \in X\ |\ [\forall S\in \mathbb{S}, x \in S \iff S \cap Y \in \mathcal{V}]\}$ is a subset of $Y$ for any ultrafilter $\mathcal{V}$ on $X$.
\end{Thm} 
\begin{cor}\label{2.3}
(see,\cite[Corollary 1.2]{Fo}) Let $X$ be a topological space satisfying the equivalent conditions of Theorem  \ref{2.2} and let $\mathbb{S}$ be as in Theorem \ref{2.2}$(ii)$. Then $\mathbb{S}$ is a subbasis of quasi-compact open subspaces of $X$.
\end{cor}

\section{$A$-subsets and Spectral spaces }\label{sec 3}
We begin this section by showing that the spectrum of a monoid, topologically speaking, is same (homeomorphic) to the spectrum of a ring, i.e., it is a spectral space.
\begin{thm}\label{mspec}
Let $MSpec(A)$ denote the set of all prime ideals of a monoid $A$. It is a topological space
when equipped with the Zariski topology whose closed sets are given by 
$V(I):=\{\mathfrak{p} \in MSpec(A)\ |\ I \subseteq \mathfrak{p}\}$ for an ideal $I$ of $A$. Then, $MSpec(A)$ is a spectral space.
\end{thm}
\begin{proof}
Clearly, a subbasis of open sets for this topology is given by $\{\mathcal{D}(x)\ |\ x \in A\}$ where $\mathcal{D}(x):=\{\mathfrak{p} \in MSpec(A)\ |\ x \notin \mathfrak{p}\}$. Further, since $\mathcal{D}(x)\cap \mathcal{D}(y)= \mathcal{D}(xy)$ for any $x, y \in A$, the collection $\{\mathcal{D}(x)\ |\ x \in A\}$ gives a basis for the Zariski topology on $MSpec(A)$. For any $\mathcal{D}(x)$ containing the unique maximal ideal $m_A$ of $A$, we have $\mathcal{D}(x) = MSpec(A)$ and this shows that $MSpec(A)$ is quasi-compact. Since this implies that $MSpec(A)$ is quasi-compact for any monoid $A$, we have $\mathcal{D}(x)$, which is same as $MSpec(A_x)$, is quasi-compact for all $x \in A$. Hence, $\{\mathcal{D}(x)\ |\ x \in A\}$ gives a basis of quasi-compact open subspaces that is clearly closed under
finite intersections. It is also easy to see that $MSpec(A)$ with this topology is $T_0$. Let $Y$ be an irreducible closed subspace of $MSpec(A)$. Then, $Y$ can be partially ordered by inclusion, i.e.,
$\mathfrak{p}\leq \mathfrak{q} \iff \mathfrak{p}\subseteq \mathfrak{q} \iff \mathfrak{q}\in \overline{\{\mathfrak{p}\}}$. Consider any arbitrary non-empty chain
of prime ideals $\{\mathfrak{p}_\alpha\}_{\alpha\in \Lambda}$ in $Y$. Since $Y$ is a closed subset, $Y=V(I)$ for some ideal $I$ of $A$. Thus, $I\subseteq \mathfrak{p}_\alpha$ for all $\alpha \in \Lambda$. Hence, $I\subseteq \cap_{\alpha\in \Lambda}\mathfrak{p}_\alpha$ and therefore $\cap_{\alpha\in \Lambda}\mathfrak{p}_\alpha$, which is a prime ideal, is in $Y$. Thus, by Zorn's lemma, $Y$ has a minimal element and let us denote it by $\mathfrak{p}_Y$. Then, $\overline{\{\mathfrak{p}_Y\}}= Y$ because otherwise it will contradict the irreducibility of $Y$. Hence, $\mathfrak{p}_Y$ is a generic point of $Y$ and it is unique because of being minimal. Thus, $MSpec(A)$ is a spectral spaces by Hochster's characterization.
\end{proof}

Let $S$ be any set. We denote the power set of $S$ as $2^S$. The topology on $2^S$ given by an open sub-basis of sets of the form, $D(F):=\{R\subseteq S~|~F\nsubseteq R\}$ where $F$ is a finite subset of $S$, is called the {\it hull-kernel topology} on $2^S$. The complement of the set $D(F)$ is denoted as $V(F):=\{R\subseteq S~|~F\subseteq R\}$. Given any $A$-set $M$, let $SSet(M|A)$ denote the set of all $A$-subsets of $M$. By the hull-kernel topology on $SSet(M|A)$, we will denote the subspace topology induced by the hull-kernel topology on $2^M$. We will show that $SSet(M|A)$ is a spectral space. This is analogous to the result of Finocchiaro, Fontana and Spirito \cite[Proposition 2.1]{Fo} for the case of submodules of a module over a ring.\smallskip
\begin{thm}\label{3.2}
For any monoid $A$ and an $A$-set $M$, $SSet(M|A)$ is a spectral space. Moreover, the collection of sets $\mathbb{S}:=\{D(x_1,\hdots,x_n)\ | \ x_1,\hdots,x_n \in M \}$
is a subbasis of quasi-compact open subspaces of $SSet(M|A)$.
\end{thm}
\begin{proof}
Let $X$ denotes $SSet(M|A)$. Let $N,N'\in X$ such that $N\neq N'$. Suppose, without loss of generality, that there exists some $y \in N$ such that $y \notin N'$. Then $SSet(M|A) \setminus V(y)$ is an open neighborhood of $N'$ not containing $N$. This shows that the hull-kernel topology on $X$ is $T_0$. By Theorem \ref{2.2}, it is now enough to show that there is a subbasis $\mathbb{S}$ of $X$ such that $X_{\mathbb{S}}({\mathcal{U}}) \neq \emptyset$ for any ultrafilter $\mathcal{U}$ on $X$. For a given ultrafilter $\mathcal{U}$ of $X$, set $N_{\mathcal{U}}:= \{y \in M\ |\ V(y) \in \mathcal{U}\}$. For $a \in A$ and $y \in N_{\mathcal{U}}$, we have $V(y) \subseteq V(ay)$ and hence $ay \in N_{\mathcal{U}}$ (Definition \ref{filter}(iii)). Hence, $N_{\mathcal{U}}$ is an $A$-subset of $M$. We claim that $N_{\mathcal{U}} \in X_{\mathbb{S}}(\mathcal{U})$ where $X_{\mathbb{S}}(\mathcal{U})$ is as in Theorem \ref{2.2}. Indeed, if $N_{\mathcal{U}} \in S=D(x_1,\hdots,x_n)$ for some $S \in \mathbb{S}$, then there exists at least one $x_i \notin N_{\mathcal{U}}$ for $i=1,\hdots,n$. Therefore, by the definition of $N_{\mathcal{U}}$, $V(x_i) \notin \mathcal{U}$ and this implies $D(x_i) \in \mathcal{U}$. Since $D(x_i) \subseteq D(x_1,\hdots,x_n)$, we have $S= D(x_1,\hdots,x_n) \in \mathcal{U}$ (Definition \ref{filter}(iii)). \\
Conversely, let $S = D(x_1,\hdots,x_n) \in \mathcal{U}$. If possible, let $N_{\mathcal{U}} \notin  
D(x_1,\hdots,x_n)$. Then $N_{\mathcal{U}}\in V(x_1,\hdots,x_n)$ i.e., $V(x_i) \in \mathcal{U}$ for all $i$. Since $\mathcal{U}$ is a filter, $\bigcap^n_{i=1} V(x_i) = V(x_1,\hdots,x_n) \in \mathcal{U}$. Hence, both $D(x_1,\hdots,x_n)$ and $V(x_1,\hdots,x_n)$ belongs to $\mathcal{U}$ and so does their intersection, i.e., $\emptyset \in \mathcal{U}$. This contradicts the definition of a filter. Thus $N_{\mathcal{U}} \in  
D(x_1,\hdots,x_n)$ and our claim is proved. Thus, $X$ is a spectral space and the sets in $\mathbb{S}$ gives a subbasis of quasi-compact open subspaces of $X$ by Corrolary \ref{2.3}. \\
\end{proof}
The next proposition is an analogue of the result \cite[Proposition 2.4]{Fo} for the case of modules over a ring.
\begin{thm}\label{3.3}
Let $M$ be an $A$-set. Then, $SSet_{\bullet}(M|A) := SSet(M|A) \setminus {M}$ endowed with the hull-kernel subspace topology
is a spectral space if and only if $M$
is finitely generated.
\end{thm}
\begin{proof}
Assume that $M$ is a finitely generated $A$-set with a finite set of generators, which we denote by $F$. A subbasis $\mathbb{S}$ for $X:= SSet_{\bullet}(M|A)$ is given by the open sets $\{ D(x_1,...,x_n)\ | \  x_1,....,x_n \in M \}$. If $\mathcal{U}$ is an 
ultrafilter on $X$, the subset $N_{\mathcal{U}}:= \{y \in M\ |\ V(y)\cap X \in \mathcal{U}\}$ is an $A$-subset of $M$ following the proof of Proposition \ref{3.2}. If only we can show that $N_{\mathcal{U}}$ is a proper subset of $M$, it follows immediately that $N_{\mathcal{U}} \in X_{\mathbb S}(\mathcal{U})$ proving that $X$ is spectral by Theorem \ref{2.2}. If $N_{\mathcal{U}} = M$, then $V(F) \cap X \in \mathcal{U}$ and as the empty set is not a member of any filter, we must have $V(F) \cap X \neq \emptyset$. So, pick any $A$-subset $N \in V(F) \cap X$. But then $N$ contains $F$, the generating set of $M$. This shows that $M = N \in \mathcal{U}$ giving a contradiction. Thus, we have $N_{\mathcal{U}} \neq M$.\smallskip
\newline Conversely, assume that $M$ is not finitely generated. Then the family of open subsets $\{D(x)\ |\ x \in M\}$ gives a cover of $X$ such that for any finite
subset $F$ of $M$, the collection of open sets $\{D(x)|x \in F\}$ is not a subcover of $X$. Indeed, consider the $A$-subset $N := \langle F \rangle $ of $M$. Since $M$ is not finitely generated, clearly $N\neq M$. Thus, $N \in X \setminus \bigcup\{D(x)\ |\ x \in F\}$. This shows that $X$ is not quasi-compact and hence not a spectral space if $M$ is not finitely generated. This completes the proof.
\end{proof}
As a particular case of Propositions \ref{3.2} and \ref{3.3}, we have the following corollary.
\begin{cor}\label{allideal}
The set of all ideals of $A$,  $Id(A):=SSet(A|A)$ and the set of all proper ideals of $A$, $Id(A)_{\bullet}:=SSet(A|A)\setminus \{A\}$, are spectral spaces.
\end{cor}

\section{Closure operations on monoids and spectral spaces}
In commutative ring theory, closure operations are defined on ideals of a ring or more generally on submodules of a module over a ring (see, for instance, \cite{N}). We begin this section by defining closure operations for monoids likewise.
\begin{defn}
Given a monoid $A$ and an $A$-set $M$, a closure operation on $SSet(M|A)$ is a
map 
\begin{equation*}
cl:SSet(M|A) \longrightarrow SSet(M|A)\ \ \ \ N \mapsto N^{cl}
\end{equation*}
satisfying the following conditions:\smallskip
\newline $(i)$~(Extension) $N \subseteq N^{cl}$ for all $N \in SSet(M|A)$.\\
$(ii)$~(Order-preservation) If $N_1 \subseteq N_2$, then $N_1^{cl} \subseteq N_2^{cl}$ for all $N_1,N_2 \in SSet(M|A)$.\\
$(iii)$~(Idempotence) $N^{cl}=(N^{cl})^{cl}$ for all $N \in SSet(M|A)$.\smallskip
\newline A closure operator $cl$ is said to be {\it finite type} if 
$N^{cl} := \bigcup \{L^{cl}\ |\  L \subseteq N,\  L \in SSet(M|A),\  L  \text{  is finitely generated } \}$.
\end{defn}
Given any closure operation on submodules of a module over a ring, a closure operation of finite type can be constructed from it (see, for instance \cite[Construction 3.1.6]{N}). We prove the same for $A$-subsets of an $A$-set over a monoid $A$ in the following proposition.
\begin{thm}
Let $cl:SSet(M|A) \longrightarrow SSet(M|A)$ be a closure operation on $SSet(M|A)$. We set,
$$ N^{cl_f} := \bigcup \{L^{cl}\ |\  L \subseteq N,\  L \in SSet(M|A),\  L  \text{  is finitely generated } \}$$
Then $cl_f$ is a closure operator of finite type. Further, $cl$ is of finite type if and only if $cl = cl_f$.
\end{thm} 
\begin{proof}
Take any $x \in N$ and consider the finitely generated $A$-subset $\langle x \rangle$ of $N$. Since $\langle x \rangle \subseteq \langle x \rangle^{cl}$, we have $x \in N^{cl_f}$. This shows extension. Order-preservation is obvious. As for idempotence, suppose $x \in (N^{cl_f})^{cl_f}$. Then there exists some finitely generated $A$-subset $L \subseteq N^{cl_f}$ such that $x \in L^{cl}$. Let $\{l_1, l_2,...,l_n\}$ be a finite generating set of $L$. Since each $l_i \in  N^{cl_f}$, there exists finitely generated $K_i \subseteq N$ such that $l_i \in {K_i}^{cl}$. We set, $K := \bigcup_{i=1}^{n} K_i$. Clearly, $K$ is finitely generated. Any element of $L$ is of the form $\lambda l_i$ for some $\lambda \in A$ and $l_i$ belonging to the generating set of $L$ and therefore $\lambda l_i \in {K_i}^{cl} \subseteq K^{cl}$. In other words, $L \subseteq K^{cl}$. Hence,
$ x \in L^{cl} \subseteq (K^{cl})^{cl} = K^{cl}$ and this implies $x \in N^{cl_f}$. So we have, $(N^{cl_f})^{cl_f} \subseteq N^{cl_f}$. The other way, $N^{cl_f} \subseteq (N^{cl_f})^{cl_f}$, follows from the extension property of $cl_f$. Thus, we have shown idempotence. Hence, $cl_f$ is a closure operation. The last statement of the lemma follows from the definition of finite type closure operation.
\end{proof}
The next proposition is an analogue of the result, \cite[Proposition 3.4]{Fo} on finite type closure operations on modules over rings.
\begin{thm}\label{3.8}
Let $M$ be an $A$-set and $cl$ be a closure operation of finite
type on $SSet(M|A)$. Then, the set $SSet^{cl}(M|A) := \{N \in SSet(M|A)\ |\  N = N^{cl}\}$ is a spectral space. Moreover, $SSet^{cl}(M|A)$ is closed in $SSet(M|A)$, endowed with the
patch topology. 
\end{thm}

\begin{proof}
Let $\mathcal{U}$ be an ultrafilter on $SSet^{cl}(M|A)$. It is enough to show that $N_{\mathcal{U}}$ (as in Proposition \ref{3.2}) is in $SSet(M|A)^{cl}$ to prove that $SSet^{cl}(M|A)$ is spectral (by Theorem \ref{2.2}). Let $x \in (N_{\mathcal{U}})^{cl}$. Since $cl$ is a closure operation of finite type, there exists a finitely generated
$A$-set $\langle{l_1,\hdots,l_n}\rangle \subseteq N_{\mathcal{U}} $ such that $x \in \langle{l_1,\hdots,l_n}\rangle^{cl}$. Hence, $x \in K^{cl}$ for all $K \supseteq \langle{l_1,\hdots,l_n}\rangle$. In other words, $x \in K^{cl}$ for all $K \in V(l_1,\hdots,l_n)$. Therefore, $V(l_1,\hdots,l_n) \cap SSet^{cl}(M|A) \subseteq V(x) \cap SSet^{cl}(M|A)$. Since $V(l_i) \cap SSet^{cl}(M|A) \in \mathcal{U}$ for all $l_i$ (by the definition of $N_{\mathcal{U}}$), we have $\cap_{i=1}^n\big(V(l_i)\cap SSet^{cl}(M|A)\big) =V(l_1,\hdots,l_n) \cap SSet^{cl}(M|A) \in \mathcal{U}$. Hence, $V(x) \cap SSet^{cl}(M|A) \in \mathcal{U}$ and so $x \in N_{\mathcal{U}}$. This shows $(N_{\mathcal{U}})^{cl} \subseteq N_{\mathcal{U}}$ and thus $N_{\mathcal{U}} =(N_{\mathcal{U}})^{cl}$. Hence, $SSet^{cl}(M|A)$ is a spectral space. Finally, it follows from Theorem \ref{2.2}(2) that $SSet^{cl}(M|A)$ is closed in $SSet(M|A)$, endowed with the
patch topology. 
\end{proof}
Traditionally closure operations for rings are defined on a whole class of rings rather than one ring at a time (see, for instance, \cite{N}). Similarly, we can define closure operations on a class of monoids and define an important property of closure operations called {\it persistence} :
\begin{defn}\label{Def3.8}
Let $\mathcal{M}$ be a subcategory of the category of monoids. Let $cl$ be a closure operation defined on the monoids of $\mathcal{M}$. We say that $cl$ is {\it persistent} if for any morphism of monoids $\phi: A \longrightarrow B$ in $\mathcal{M}$ and any ideal $I$ of $A$, we have $\phi(I^{cl})B \subseteq (\phi(I)B)^{cl}$. For a persistent closure
operation $cl$ on a subcategory $\mathcal{M}$ of monoids, we say that $cl$ commutes with localization
in $\mathcal{M}$ if for any monoid $A \in \mathcal{M}$ and any multiplicative set $S \subseteq A$, such that the localization map $A \longrightarrow S^{-1}A$ is in  $\mathcal{M}$, we have
$S^{-1}(I^{cl}) = (S^{-1}I)^{cl}$
for every ideal $I$ of $A$.
\end{defn} 
\begin{ex}
\emph{ Simple examples of closure operations on $\mathcal{I}$, the collection of all ideals of a monoid $A$, are:}\smallskip
\newline\emph{(1) ({\it Identity closure}) $I^{cl} := I$ for all $I \in \mathcal{I}$.}
\smallskip
\newline\emph{(2) ({\it Indiscrete closure}) $I^{cl} := A$ for all $I \in \mathcal{I}$.}
\smallskip
\newline\emph{(3) We define the {\it radical closure} of an ideal $I \in \mathcal{I}$ as $I^{cl} := \sqrt I$ for all $I \in \mathcal{I}$ where $ \sqrt I$ is the intersection of all the prime ideals of $A$ containing $I$ or equivalently it is the set $\{x \in A\ |\ \exists\ n \in \mathbb{N} \textnormal{\ \ such that \ } x^n \in I \}.$}\smallskip
\newline \emph{It is easy to show that identity, indiscrete and radical closures are persistent on the category of all monoids and commutes with localization.}
\end{ex}
We now introduce some interesting closure operations on monoids like  Frobenius, tight, integral and saturation closure operations, which are inspired by such similar operations for rings in classical commutative algebra. We begin by introducing Frobenius closure of ideals over monoids. See, for instance, \cite{N} for details on Frobenius closure of ideals over rings. \smallskip
\newline Given an ideal $I$ of a monoid $A$, let $I^n:=\{\prod_{i=1}^nx_i~|~x_i\in I\}$ and $I^{[n]}:=\{ax^n~|~a\in A, x\in I\}$.
\begin{thm}
Let $I \in \mathcal{I}$. We set,
$$I^{Frob} := \{ x \in A\ |\ \exists \ n \in \mathbb{N} \textnormal{\ \ such that\ } x^n \in I^{[n]} \}$$ 
(1) This defines a closure operation on $\mathcal{I}$ and we call this the {\it Frobenius closure\footnote{We call this the ``Frobenius closure" closure because it is motivated by the usual Frobenius closure for rings of characteristic $p > 0$. If $R$ is a ring of characteristic $p>0$, the association $a \mapsto a^{p^e}$ ($e \in \mathbb{N}$) defines a ring endomorphism. Clearly, this does not hold in general for arbitrary $n \in \mathbb{N}$. However, since the additive structure does not arise in the case of monoids, we can consider here a Frobenius closure for any $n$.}} on $\mathcal{I}$. \smallskip
\newline(2) Frobenius closure is persistent on the category of all monoids and commute with localization.
\end{thm}
\begin{proof}
Since extension and order-preservation are obvious, it is enough to prove the idempotence condition for $(1)$. Take $x \in ({I^{Frob}})^{Frob}$. Then, $x^n \in ({I^{Frob}})^{[n]}$ for some $n \in \mathbb{N}$, i.e., $x^n = ay^n$ for some $a\in A$ and $y \in I^{Frob}$. Again this implies $y^k \in I^{[k]}$ for some $k \in \mathbb{N}$, i.e., $y^k = a'z^k$ for some $a'\in A$ and $z \in I$. Hence,
$x^{nk} = (x^n)^k = a^k(y^k)^n = a^k(a'z^k)^n = a^k a'^n z^{nk}$.
Therefore, $x^{nk} \in I^{[nk]}$ or in other words $x \in I^{Frob}$. The other inclusion is obvious.\smallskip
\newline Let $\phi$ and $S$ be as in Definition \ref{Def3.8}. Consider any $\phi(x)$ for $x \in I^{Frob}$. Then, $x^n \in I^{[n]}$ for some $n \in \mathbb{N}$. Thus, $x^n = ay^n$ for some $y \in I$ and $a \in A$. Hence, $ (\phi(x))^n = \phi(x^n) = \phi(ay^n) = \phi(a)\phi(y)^n \in (\phi(I)B)^{[n]}$. This shows persistence. Now, pick any $x/s \in S^{-1}(I^{Frob})$ where $x \in I^{Frob}$ and $s \in S$. Since $x \in I^{Frob}$, there exists some $n \in \mathbb{N}$ such that $x^n \in I^{[n]}$. Thus, $x^n = ay^n$ for some $a \in A$ and $y \in I$. Since, $(x/s)^n = x^n/s^n = ay^n/s^n = (a/1)(y/s)^n \in (S^{-1}I)^{[n]}$, we have $x/s \in (S^{-1}I)^{Frob}$. This proves (2).
\end{proof}
Tight closure of ideals over rings was introduced by M. Hochster and C. Huneke in 1986 \cite{HH}. Inspired by the same, we introduce tight closure of ideals over monoids. 
\begin{thm}
 Let $A$ be a Noetherian cancellative monoid and $I\in \mathcal{I}$. Then, $I$ must be finitely generated and let $\{y_1,\hdots,y_r \}$ denote a finite set of generators of $I$. We set,
$$ I^{tight} := \{x \in A\ |\ \exists\ 0 \neq a \in A \textnormal{\ such that\ } ax^n \in I^{[n]}\  \forall\  n\gg 0 \}.$$
(1) This defines a closure operation on $\mathcal{I}$ and we call this the tight closure operation.\smallskip
\newline(2)  Tight closure is persistent on the subcategory $\mathcal{M}$ of Noetherian cancellative monoids whose morphisms have zero kernel.
\end{thm}
\begin{proof}
Note that any element of $I^{[n]}$ is of the form $ay^n$ where $a \in A$ and $y \in I$. Again, $y$ must be of the form $a'y_i$ for some $a' \in A$ and $y_i \in \{y_1,\hdots,y_r \}$. Henceforth, $I^{[n]} = \langle{y_1^n,\hdots, y_r^n }\rangle$. Let $\{z_1,\hdots,z_k \}$ be a finite set of generators of the ideal $I^{tight}$. Pick any $x \in (I^{tight})^{tight}$. Then, there exists $a \in A$ such that $ax^n \in \langle z_1^n,\hdots,z_k^n \rangle$ for all $n\gg 0$, i.e, $ax^n = a'z_i^n$ for some $a' \in A$ and for some $z_i \in \{z_1,\hdots,z_k \}$. Since $z_i \in I^{tight}$, there exists some $0 \neq a'' \in A$ such that $a''z_i^{n} \in \langle y_1^n,...,y_r^n \rangle$ for all $n\gg 0$. Therefore, $a''ax^n = a''a'z_i^n \in \langle y_1^n,...,y_r^n \rangle$ for all $n\gg 0$. Note that $a''a \neq 0$ since both $a'', a \neq 0$ and $A$ is a cancellative monoid. Thus, we have $x \in I^{tight}$. Since $I^{tight} \subseteq (I^{tight})^{tight}$ is obvious, we have $I^{tight} = (I^{tight})^{tight}$. This shows idempotence. \smallskip
\newline Let $\phi \in \mathcal{M}$. Consider any $\phi(x)$ for $x \in I^{tight}$. Then, there exists $0\neq a\in A$ such that $ax^n \in I^{[n]}\ \textnormal{forall}\ n\gg 0$.  This implies $\phi(a)\phi(x)^n=\phi(ax^n) \in \phi(I^{[n]}) \subseteq\phi(I^{[n]})B = (\phi(I)B)^{[n]} \ \textnormal{for all}\ n\gg 0$. But, $\phi(a)\neq 0$ by assumption. This shows persistence.\smallskip
\end{proof}

A classical study of integral closure of ideals over rings can be found, for instance, in \cite{HS}. We now introduce integral closure of ideals over monoids.
\begin{thm}
Let $I\in \mathcal{I}$. We set,
$$I^{int} := \{x \in A\ |\ \exists\ n \in \mathbb{N} \textnormal{\ \ such that \ } x^n \in I^n \}$$
(1) This defines a closure operation on $\mathcal{I}$ and we call this the {\it integral closure}.\smallskip
\newline(2) Integral closure is persistent on the category of all monoids and commutes with localization.
\end{thm}
\begin{proof}
Since extension and order-preservation are obvious, we only show idempotence for (1). The inclusion $I^{int} \subseteq (I^{int})^{int}$ follows immediately from extension. To show the other way, pick any $x \in (I^{int})^{int}$. Then, there exists some $n \in \mathbb{N}$ such that $x^n \in (I^{int})^n$. Thus,
$ x^n = \prod_{i=1}^n y_i$ for some $y_i \in I^{int}$. Now,
$y_i \in I^{cl} \implies y_i ^{k_i} \in I^{k_i}$ for some $k_i \in \mathbb{N}$. If we set, $k = \prod_{i=1}^n k_i$, then clearly $y_i^k \in I^k$ for all $y_i$. Therefore, $x^{nk} = \prod_{i=1}^n y_i ^{k} \in I^{nk}$ and so $x \in I^{int}$.\smallskip
\newline Let $\phi$ and $S$ be as in Definition \ref{Def3.8}. Consider any $\phi(x)$ for $x \in I^{int}$. Then, $x^n \in I^n$ for some $n \in \mathbb{N}$. Hence,
$ (\phi(x))^n = \phi(x^n) \in \phi(I^n) \subseteq \phi(I^n)B = (\phi(I)B)^n$. This shows persistence. Now, pick any $x/s \in S^{-1}(I^{int})$ where $x \in I^{int}$ and $s \in S$. Since $x \in I^{int}$, there exists some $n \in \mathbb{N}$ such that  $x^n \in I^n \subseteq I$. Thus, $(x/s)^n = x^n/s^n \in {S^{-1}I}$ which implies $x/s \in (S^{-1}I)^{int}$. Hence, $S^{-1}(I^{int}) \subseteq (S^{-1}I)^{int}$. For the other way, let $y/t \in (S^{-1}I)^{int}$ for $y \in A$ and $t \in S$. Then, $(y/t)^n \in (S^{-1}I)^n$ for some $n \in \mathbb{N}$. Thus,
$(y/t)^n = \prod_{i=1}^ny_i/t_i$ for some $y_i \in I$ and $t_i \in S$. Hence, there exists some $u \in S$ such that $uy^n\Big(\prod_{i=1}^n t_i\Big) = ut^n\Big(\prod_{i=1}^n y_i\Big)$. Therefore, $\Big(uy\Big(\prod_{i=1}^n t_i\Big)\Big)^n = u^ny^n\Big(\prod_{i=1}^n t_i^n\Big) = u^nt^n\Big(\prod_{i=1}^n y_i\Big)\Big(\prod_{i=1}^n t_i^{n-1}\Big) \in I^n$ as $y_i \in I$. Thus, $y/t = uy\Big(\prod_{i=1}^n t_i\Big)/ut\Big(\prod_{i=1}^n t_i\Big) \in S^{-1}(I^{int})$. This proves (2).
\end{proof}
See, for instance, \cite{N} for details on saturation closure of ideals over rings. Here we introduce saturation closure of ideals over monoids.
\begin{thm}
Fix a finitely generated ideal $\mathfrak{a} \in \mathcal{I}$. Then for any $I \in \mathcal{I}$, we set
$$ I^{\mathfrak{a}-sat} := \bigcup_{n \in N}(I : \mathfrak{a}^n) = \{ x \in A\ |\ \exists\ n \in \mathbb{N} \textnormal{\ \ such that \ } \mathfrak{a}^{n}x \subseteq I \}$$
(1) This defines a closure operation on $\mathcal{I}$ and we call this the $\mathfrak{a}$-{\it saturation} closure operation.\smallskip
\newline(2) Consider the subcategory $\mathcal{M}$ of Noetherian monoids whose morphisms are surjective\footnote{Note that the saturation closure with respect to the unique maximal ideal of each monoid is not persistent on the subcategory of Noetherian monoids and local morphisms without the assumption of surjection. For example, consider the inclusion $\{0,1\}\xrightarrow{i} A$ where $A$ is any Noetherian cancellative monoid. Then the saturation closure of the ideal $\langle 0\rangle$ in $\{0,1\}$ is $\{0,1\}$ and $i(\{0,1\}) = \{0,1\}$. But, the saturation closure of the ideal $\langle i(0)\rangle$ is just $\{0\}$ since $A$ is cancellative. } local morphisms.  Saturation closure with respect to the unique maximal ideal of each monoid is persistent on this category and commutes with localization.
\end{thm}
\begin{proof}
We only show idempotence for (1). Take any $x \in {(I^{{\mathfrak{a}-sat}})}^{{\mathfrak{a}-sat}} $.
Then, there exists some $n_x \in \mathbb{N}$ such that $\mathfrak{a}^{n_x}x \subseteq I^{{\mathfrak{a}-sat}} =\bigcup_{n \in N}(I : \mathfrak{a}^n)$. 
Since $\mathfrak{a}$ is finitely generated, so is $\mathfrak{a}^{n_x}x$. Let $S$ denote a finite generating set of $\mathfrak{a}^{n_x}x$. Then, for each generator $s \in S \subseteq \mathfrak{a}^{n_x}x$, there exists $n_s \in \mathbb{N}$ such that $\mathfrak{a}^{n_s}s \subseteq I$. We set $N_S := \prod_{s \in S} n_s \in \mathbb{N}$. Then, for every $s \in S$, $\mathfrak{a}^{N_S}s \subseteq \mathfrak{a}^{n_S}s \subseteq I$.
Consequently, $\mathfrak{a}^{n_x +N_S}x \subseteq I$.
Hence, $x \in I^{{\mathfrak{a}-sat}}$. The other inclusion $I^{{\mathfrak{a}-sat}} \subseteq (I^{{\mathfrak{a}-sat}})^{{\mathfrak{a}-sat}}$ is obvious.\smallskip
\newline Consider any $x \in I^{{\mathfrak{a}-sat}}$. Then, there exists some $n \in \mathbb{N}$ such that $(m_A)^n x \subseteq I $. For any morphism $\phi:A \longrightarrow B$ in $\mathcal{M}$, we have, $\phi(m_A) = m_B$ where $m_B$ denote the maximal ideal of $B$. Hence, $(m_B)^n\phi(x)= \phi(m_A)^n\phi(x) = \phi((m_A)^nx) \subseteq \phi(I)$. This shows persistence. Let $m_A$ denote the maximal ideal of a monoid $A \in \mathcal{M}$. For an ideal $I$ of $A$, the $m_A$-{\it saturation} closure of $I$ is defined as
$\{ x \in A\ |\ \exists\ n \in \mathbb{N} \textnormal{\ \ such that \ } (m_A)^{n}x \subseteq I \}$. Consider any multiplicative set $S \subseteq A$ such that the localization $A \longrightarrow S^{-1}A$ is in $\mathcal{M}$. This implies that $S^{-1}(m_A)$ is the unique maximal ideal of $S^{-1}A$. Pick any $x/s \in S^{-1}(I^{{\mathfrak{a}-sat}})$ for $x \in I^{{\mathfrak{a}-sat}}$ and $s \in S$. Then, there exists some $n \in \mathbb{N}$ such that $(m_A)^{n}x \subseteq I$. Clearly, this implies $(S^{-1}(m_A))^n(x/s) \subseteq S^{-1}I$ and hence, $S^{-1}(I^{{\mathfrak{a}-sat}}) \subseteq (S^{-1}I)^{{\mathfrak{a}-sat}}$. To show the other way, take any $y/t \in (S^{-1}I)^{{\mathfrak{a}-sat}}$ for $y \in A$ and $t \in S$. Then, there exists some $n \in \mathbb{N}$ such that $(S^{-1}(m_A))^n(y/t) \subseteq S^{-1}I$. Since $A$ is Noetherian, $(m_A)^n$ is finitely generated and let $\{a_1,\hdots,a_k\}$ be a finite set of generators of $(m_A)^n$. Clearly, $a_i/1 \in (S^{-1}(m_A))^n$ and hence $(a_i/1)(y/t) \in S^{-1}I$. This implies that there exists some $s_i \in S$ such that $a_iys_i \in I$. Now, set $y' :=y(\prod_{i=1}^ks_i)$. It is easy to see that $(m_A)^ny' \subseteq I$. If $s:= \prod s_i$, then we have, $y/t = y'/st \in S^{-1}(I^{{\mathfrak{a}-sat}})$. 

\end{proof}

\section{Continuous valuation and Spectral spaces}\label{sec 4}
Other than rings, valuation has also been studied on ring-like objects like semirings (see, for instance \cite{JJ}, \cite{NP}). Here, we study valuation on monoids. Given any field $F$, the set of all valuation rings
having the quotient field $F$ is a spectral space (see, for instance, \cite[Example 2.2 (8)]{O}). As shown by Huber in \cite{Hu}, the valuation spectrum of any ring gives a spectral space and for an $f$-adic (or Huber) ring, the collection of continuous valuations gives a spectral space too (see also, \cite{Wd} and \cite{Con}). In this section, we prove similar results for monoids. We begin by briefly recalling totally ordered groups (resp. monoids).\medskip
\newline A totally ordered group (resp. totally ordered monoid) is a commutative
group (resp. commutative monoid) $\Gamma$ together with a total order $\leq$ on $\Gamma$ which is compatible with the group operation. 
A homomorphism of totally ordered groups (resp. totally ordered monoids) is an order preserving homomorphism of groups (resp. monoids).
Equivalently, a group homomorpism $f:\Gamma_1 \longrightarrow \Gamma_2$ is a homomorphism of totally ordered groups $\Gamma_1$ and $\Gamma_2$ if and only if for all $a \in \Gamma_1$ with $a \leq 1$, we have $f(a) \leq 1$.\smallskip
\newline Given a totally ordered abelian group (written multiplicatively) $\Gamma$, we can obtain a totally ordered monoid $\Gamma_* := \Gamma \cup \{0\}$ by adjoining a basepoint 0 to $\Gamma$ and extending the multiplication and the ordering of $\Gamma$ to $\Gamma_*$  by setting
$\gamma\cdot 0 = 0 \cdot \gamma = 0$ and $\gamma \geq 0$ for all $\gamma \in \Gamma$. From now onwards, $\Gamma_*$ will always denote this monoid corresponding to a given group $\Gamma$.

\begin{lem}\label{convex}
Let $\Gamma$ be a totally ordered group with $\gamma, \delta, \in \Gamma$. Then,\\
(1) $\gamma < 1 \iff \gamma^{-1} >1$\\
(2) $\gamma ,\delta \leq 1 \implies \gamma\delta \leq 1 \ \ ; \ \ \gamma <1, \delta \leq 1 \implies \gamma\delta <1$\ \ ;\ \
 $\gamma ,\delta \geq 1 \implies \gamma\delta \geq 1 \ \ ; \ \ \gamma >1, \delta \geq 1 \implies \gamma\delta >1$
 \smallskip
 \newline If $\Gamma'$ is a subgroup of $\Gamma$, then for all $\gamma,\delta,\eta \in \Gamma$, we have the following equivalent conditions: \smallskip
\newline$(i)$~ $\gamma \leq \delta \leq 1$ and $\gamma \in \Gamma'$ imply $\delta \in \Gamma'$.\\
$(ii)$~  $\gamma, \delta \leq 1$ and $\gamma\delta \in \Gamma'$ imply $\gamma,\delta \in \Gamma'$.\\
$(iii)$~ $\gamma \leq \delta \leq \eta$ and $\gamma, \eta \in \Gamma'$ imply $\delta \in \Gamma'$.\\
A subgroup $\Gamma'$ of $\Gamma$ satisfying one of these equivalent conditions is called {\it convex} .
\end{lem}
\begin{proof}
Omitted. See, for instance, \cite[Remark {1.6} and Remark {1.7} ]{Wd}.
\end{proof}
Let $\Gamma$ be a totally ordered group and let $\Gamma'$ be a subgroup of $\Gamma$. An element $\gamma$ of $\Gamma \cup \{0\}$ is called {\it cofinal} for $\Gamma'$ if for all $\gamma' \in \Gamma'$ there exists $n \in \mathbb{N}$ such that $\gamma^{n} < \gamma'$.
\begin{lem}\label{4.2}
Let $\Gamma$ be a totally ordered group and $\Gamma' \subseteq \Gamma$ be a convex subgroup. Let $\gamma \in \Gamma$ be cofinal for $\Gamma'$ and let $\Delta \subsetneq \Gamma'$ be a proper convex subgroup. Then $\delta\gamma$ is cofinal
for $\Gamma'$ for all $\delta \in \Delta$.
\end{lem}
\begin{proof}
Omitted. See, for instance, \cite[Proposition 1.20]{Wd}.
\end{proof}
We now briefly recall valuation monoids from \cite{CHWW}. A monoid $A$ is called a {\it valuation monoid} if $A$ is cancellative and for every nonzero
element $\alpha \in A^{+}$ (the group completion of $A$), at least one of $\alpha$ or $1/\alpha$
belongs to $A$.
For a valuation ring $(A,+,\cdot)$, the underlying multiplicative monoid $(A,\cdot)$ is a valuation
monoid. Another simple example of a valuation monoid is the free pointed monoid on one generator.\smallskip 
\newline The set of units $A^{\times}$ of a given valuation monoid $A$ is a subgroup of $A^+ \setminus 0$ and the quotient group $(A^+ \setminus 0)/ A^{\times}$ is a totally ordered abelian group with the total ordering defined by $x \geq y$ if and only if $ x/y$ belongs to the image of $A \setminus 0$. We call the totally ordered pointed monoid $\Gamma  := ((A^+ \setminus 0)/ A^{\times})_*$ the {\it value monoid} of the valuation monoid $A$ . The canonical surjection
$ord: A^{+} \twoheadrightarrow \Gamma$
is called the {\it valuation map} of $A$ . The monoid $A$ can be identified with the set of $x$ in $A^{+}$ such that $ord(x)\leq 1$ where 1 represents the identity of $\Gamma$. Note that $ord(x) =1 $ only if $x$ is in $A^{\times}$. Thus, the maximal ideal $m_A$ of $A$ is $\{x \in A^{+}\ |\ ord(x) <1\}$.
Clearly, for all $x, y \in A^{+}$,
\begin{equation}\label{ord}
 ord(x) \geq 0\ \ ; \ \ ord(xy) = ord(x)ord(y)\ \ ; \ \  ord(x) = 0 \textnormal{\ \ if and only if \ } x=0
\end{equation}
Conversely, given an abelian group $G$, consider the monoid $G_*$ obtained by adjoining a basepoint $0$. Then, for a surjective morphism $ord: G_* \twoheadrightarrow \Gamma$ onto a totally ordered monoid $(\Gamma, \cdot,1,0)$  that satisfies all the conditions in (\ref{ord}), we get a valuation monoid given by $\{ g \in G\ | \ ord(g) \leq 1 \}$ whose
group completion is $G_*$ and whose associated valuation map is $ord$.
\begin{thm}\label{vmsp}
Let $A$ be a monoid contained in a monoid $G_*$, where $G_*$ is a monoid obtained by adjoining a basepoint $0$ to an abelian group $G$. The set of all valuation monoids containing $A$
and having the group completion $G_*$, is a spectral space.
\end{thm}
\begin{proof}
For any set $S$, the power set $2^S$ endowed with the hull-kernel topology, which has as an open sub-basis the sets of the form $\mathcal{D}(F)=\{R\subseteq S~|~F\nsubseteq R\}$ where $F$ is a finite subset of $S$, is a spectral space (see,\cite[Theorem 8 and Proposition 9]{H}). The sets $\mathcal{D}(F)$ are quasi-compact and the complement of $\mathcal{D}(F)$ is denoted as $\mathcal{V}(F)=\{R\subseteq S~|~F\subseteq R\}$. Taking $S=G_*$, we see that $2^{G_*}$ is a spectral space. The set of all valuation monoids containing $A$
and having the group completion $G_*$, is clearly a subset of $2^{G_*}$ and it is given by
$$\Big(\big(\bigcap_{a\in A} \mathcal{V}(a)\big)\cap\big(\bigcap_{g_1,g_2 \in G_*}\mathcal{D}(g_1,g_2)\cap\mathcal{V}(g_1g_2)\big)\Big)\bigcap_{0\neq h\in G_*}\big(\mathcal{V}(h)\cup \mathcal{V}(h^{-1}\big).$$ It is easy to see that this is a closed set in the patch topology on $2^{G_*}$. Since any patch closed set of a spectral space is also spectral in the subspace topology (see, \cite[Proposition 9]{H} or \cite[Proposition 2.1]{O} ), we have the required result. 
\end{proof}
Note that by taking $A=\{0,1\}$ in Proposition \ref{vmsp}, we have that the set of all valuation monoids having the same group completion $G_*$ forms a spectral space.
\begin{rem}
It is worth mentioning that, like in Proposition \ref{vmsp}, using the property that a patch in a spectral space is also spectral, the results obtained in Section \ref{sec 3} can be given shorter proofs. For example, given a monoid $A$, the collection of all proper ideals of $A$ is a spectral space since it is given by the patch closed set, 
$X=D(1)\cap\Big(\bigcap_{a,b\in A} D(a)\cup V(ba)\Big)$
of $2^A$ and the collection of all
prime ideals, $MSpec(A)$, is a spectral space since it is given by the patch closed subset,
$X\cap\Big(\bigcap_{a,b\in A} D(ab)\cup V(a)\cup V(b)\Big)$
of $2^A$. This method of obtaining spectral spaces by expressing them as patches in the power set of some set, endowed with the hull-kernel topology, was used in \cite[Example 2.2]{O} for proving the ring theoretic versions of similar results. This method does not appear to extend beyond Section \ref{sec 3} of this paper.
\end{rem}
\subsection{Valuation spectrum of a monoid as a spectral space}
We begin this section by defining valuation on a monoid. 
\begin{defn}\label{val}
Let $A$ be a monoid and $\Gamma$ be a totally ordered abelian group written multiplicatively. Let $\Gamma_* = \Gamma \cup\{0\}$ denote the corresponding totally ordered monoid. We call a mapping $v: A \longrightarrow \Gamma_*$ is a valuation of $A$ with values in $\Gamma_*$ if the following conditions are satisfied:\smallskip
\newline $(i)~ v(xy) = v(x)v(y)$ for all $x,y \in A$\\
$(ii)~ v(0) = 0$, $ v(1) = 1$\smallskip
\newline We call the subgroup of $\Gamma$ generated by $im(v)\setminus \{0\}$ the {\it value group} of $v$ and we denote it by $\Gamma_v$.
\end{defn}
Note that for a valuation monoid $A$, the corresponding valuation map $ord$ as defined by the equation (\ref{ord}) does satisfy the definition of a valuation as in Definition \ref{val}.
\begin{ex}
\emph{(1) Let $A$ be a monoid and let $\mathfrak{p}$ be a prime ideal of $A$. Then, 
\[
    a \mapsto \left\{
                \begin{array}{ll}
                  1, \ \ \text{when }a \notin \mathfrak{p}\\
                  0, \ \ \text{when }a \in \mathfrak{p}
                \end{array}
              \right.
  \]
gives a valuation with value group 1. We call every valuation of this type a {\it trivial} valuation.}\smallskip
\newline\emph{(2) Restricting any valuation on a ring $(A,+,\cdot)$ to the underlying multiplicative monoid $(A,\cdot)$ gives a valuation on the monoid.}\smallskip
\newline\emph{(3) Let $\phi: G \longrightarrow \Gamma$ be a group homomorphism from an abelian group $G$ to a totally ordered abelian group $\Gamma$. Adjoining a basepoint 0 to both the groups, we obtain a valuation $v:G_* \longrightarrow \Gamma_*$ on the monoid $G_*$ given by 
$ v(g) := \phi(g) \textnormal{\ if \ } g \neq 0$ and
$ v(0) := 0$.}\smallskip
\end{ex}
\begin{defn}
{\it Let $v$ be a valuation on a monoid $A$ and let $\Gamma_v$ be its value group. Then the characteristic subgroup of $\Gamma_v$ is the convex subgroup (convex with respect to the total ordering of $\Gamma$) generated by the set $\{v(x) \in \Gamma_v\ |\ v(x) \geq 1\}$. We denote it as $c\Gamma_v$.}
\end{defn}
It is easy to check that $c\Gamma_v$ is the set of all $\gamma \in \Gamma$ such that $v(x)^{-1} \leq \gamma \leq v(x)$ for some $x \in A$ with $v(x) \geq 1$. 
\begin{ex}
\emph{(1) Let $A:= G_*$ be a monoid obtained from a group $G$. For any valuation $v$ on $A$, it is easy to check that $c\Gamma_v = \Gamma_v$.}\vspace{2mm}
\newline\emph{(2) Let $A$ be a valuation monoid and $v$ be the corresponding valuation map. Then, $\{v(x)\in \Gamma_v\ |\ v(x) \geq 1\} = \{1\}$ and hence $c\Gamma_v = 1$. For a trivial valuation on a monoid, we also have  $c\Gamma_v = 1$}.  
\end{ex}
\begin{lem}
Let $v$ be a valuation on a monoid $A$. Then, $supp(v):= v^{-1}(0)$ is a prime ideal.
\end{lem}
\begin{proof} For any $x \in v^{-1}(0)$ and any $a \in A$, we have $v(ax) = v(a)v(x) = v(a)0 =0$. Hence $ax \in v^{-1}(0)$ and this shows that $v^{-1}(0)$ is an ideal. To further show that this ideal is prime, take $x,y \in A$ such that $xy \in v^{-1}(0)$. Then $v(xy) = v(x)v(y) = 0$. If $v(y) \neq 0$, then $v(x) = v(x)v(y)v(y)^{-1} = 0v(y)^{-1} = 0$. Hence, $supp(v)$ is a prime ideal.
\end{proof}
Since $supp(v)$ is a prime ideal, $A/supp(v)$ is torsion free and so we can consider the group completion of $A/supp(v)$. Clearly, the valuation $v$ on $A$ induces a valuation on its quotient $A/supp(v)$ and the inverse image of $0$ under this induced valuation reduces to 0. This again induces a valuation $\bar{v}:(A/supp(v))^{+} \longrightarrow \Gamma_*$ on the group completion $(A/supp(v))^{+}$ defined by $\bar{v}(\bar{a}/\bar{b}):= v(a)v(b)^{-1}$ where $\bar{a},\bar{b} \in A/supp(v)$ are the images of elements $a,b \in A$. It can be easily checked that $\bar{v}$ is well-defined by using the multiplicativity of $v$ and the fact that $\bar{a} = \{a\}$ for every non-zero $\bar{a} \in  A/supp(v)$. Note that $\Gamma_v = \Gamma_{\bar{v}}$.
\begin{lem}
Let $v$ be a valuation on a monoid $A$. For any $x,y \in (A/supp(v))^{+}$, $\bar{v}$ satisfies the following conditions:\smallskip
\newline $(i)~ \bar{v}(x) \geq 0$\\
$(ii)~\bar{v}(xy) = \bar{v}(x)\bar{v}(y)$\\
$(iii)~ \bar{v}(x) = 0$ if and only if $x=0$
\end{lem}
\begin{proof}
Let us only check that $x \neq 0$ implies $\bar{v}(x) \neq 0$. If possible, let there exists some $0 \neq x \in (A/supp(v))^{+}$ such that $\bar{v}(x) = 0$. Then,
$1 = \bar{v}(1) = \bar{v}(xx^{-1}) = \bar{v}(x)\bar{v}(x^{-1}) = 0$. This gives the required contradiction.

\end{proof}
As discussed in the beginning of $\S$ \ref{sec 4}, we have a valuation monoid $A(v):= \{x \in (A/supp(v))^{+}\ |\ \bar{v}(x) \leq 1\}$ corresponding to the valuation $\bar{v}$. Note that $\bar{v}(x)=1$ if and only if $x$ is invertible in $A(v)$. Hence, $m_{A(v)} := \{x \in (A/supp(v))^{+}\ |\ \bar{v}(x)< 1\}$ is the maximal ideal of $A(v)$.
\begin{thm}\label{equival}
Two valuations $v$ and $w$ on a monoid $A$ are
called equivalent if the following equivalent conditions are satisfied:\smallskip
\newline $(i)$ There exists an isomorphism of totally ordered monoids, $f:(\Gamma_{v})_* \longrightarrow (\Gamma_{w})_*$, such that $w= f \circ v$.\\
$(ii)~ supp(v) = supp(w)$ and $A(v) = A(w)$.\\
$(iii)~ v(x) \leq v(y)$ if and only if $w(x) \leq w(y)$ for all $x,y \in A$.\smallskip
\newline Note that $f$ in $(i)$ restricts to an isomorphism $\Gamma_{v} \longrightarrow \Gamma_{w}$ of totally ordered groups.
\end{thm}
\begin{proof}
Since homomorphism of totally ordered monoids preserves order, $(i)$ implies $(iii)$ obviously. It is also easy to see that $(iii)$ implies $(ii)$. Let us now assume $(ii)$. 
Let $\mathfrak{p}:= supp(v) = supp(w)$. Then, we have induced valuations $\bar{v}$ and $\bar{w}$ on $(A/\mathfrak{p})^+$. We claim that $\{x \in (A/\mathfrak{p})^+\ |\ \bar{v}(x) \geq 1\} = \{x \in (A/\mathfrak{p})^+\ |\ \bar{w}(x) \geq 1\} $.
Indeed, let $x \in (A/\mathfrak{p})^+$ such that $\bar{v}(x) \geq 1$. If $\bar{v}(x) > 1$, then $x \notin A(v) = A(w)$ and so $\bar{w}(x) >1$. Now, if $\bar{v}(x) = 1$, we claim that $\bar{w}(x)$ must also be 1. This is because if $\bar{w}(x) \neq 1$, then either $\bar{w}(x)>1$ or $\bar{w}(x) <1 $. If $\bar{w}(x)>1$, then $x \notin A(w) = A(v)$ and so $\bar{v}(x) >1$ which gives a contradiction. If $\bar{w}(x)<1$, then $x \in m_{A(w)} = m_{A(v)}$, i.e.,  $\bar{v}(x) <1$ which also gives a contradiction. Hence, $\bar{w}(x) = 1$ and this proves our claim.\smallskip
\newline This also shows that $\bar{v}$ and $\bar{w}$ restrict respectively to the surjective group homomorphisms
$$\bar{v}: (A/\mathfrak{p})^+ \setminus \{0\} \longrightarrow \Gamma_v \ \  \textnormal{and} \ \ \bar{w}: (A/\mathfrak{p})^+ \setminus \{0\} \longrightarrow \Gamma_w$$
with the same kernel. Thus, there exists a unique
group homomorphism $f :\Gamma_v \longrightarrow \Gamma_w$ such that $f \circ \bar{v} = \bar{w}$. It can be easily checked that $f$ is an isomorphism. Note that $f$ maps elements $\leq 1$ to elements $\leq 1$  and hence it is a homomorphism of totally ordered groups. We can
extend $f$ to $(\Gamma_{v})_*$ by setting $f(0) = 0 $. Finally, it is easy to see that $f \circ \bar{v} = \bar{w}$ implies $f \circ v = w $. This proves $(i)$.
\end{proof}
\noindent From now onwards, we will not differentiate between a valuation and its equivalence class. We are now all set to define the valuation spectrum of a monoid.
\begin{defn}
{\it Let $A$ be a monoid. The {\it valuation spectrum} $Spv( A)$ is the set of all
equivalence classes of valuations on $A$ equipped with the topology generated by the
subsets
$$Spv(A)\Big(\frac{x}{y}\Big):= \{v \in Spv(A)\ | \ v(x) \leq v(y) \neq 0 \}$$ where $x,y \in A$.
}
\end{defn}
We now proceed to show that $Spv(A)$ is a spectral space for which we first discuss some preparatory lemmas.
\begin{lem}\label{Step1}
Given two valuations $v \neq w $ in $Spv(A)$ there exists some open set $U = Spv(A)(\frac{x}{y})$ with $x,y \in A$ such that either $v \in U$ but $w \notin U$, or $v \notin U$ but $w \in U$. In other words, $Spv(A)$ is a $T_0$ space.
\end{lem}
\begin{proof}
Since $v$ and $w $ are two non-equivalent valuations on $A$, there exists $x,y \in A$ such that $v(x) \leq v(y)$ but $w(x) \nleq w(y)$ (or conversely) by Proposition \ref{equival}$(iii)$. If $v(y)\neq 0$, then there exists the open set $U=Spv(A)\Big(\frac{x}{y}\Big)$ such that $v\in U$ but $w\notin U$. Now, if $v(y)=0$, then we also have $v(x)=0$. Since $w(y)\lneq
 w(x)$, clearly $w(x)\neq 0$. Thus, there exists the open set $V= Spv(A)\Big(\frac{y}{x}\Big)$ such that $w\in W$ but $v\notin V$. This  finishes the proof.
\end{proof}
To the equivalence class $v$ of a valuation on $A$ we attach a binary relation $|_v$ on $A$ by defining $x|_v y$ if and only if $v(x) \leq v(y)$. This gives us a map, $$\phi: Spv(A) \longrightarrow P(A \times A)\ \ \ \ v \mapsto |_v$$ into the power set $P(A \times A)$ of $A \times A$. We identify $P(A \times A)$ with $\{0,1\}^{A\times A}$ and endow it with the product topology, where $\{0,1\}$ carries the discrete topology. Thus, $P(A \times A)$ is compact (by Tychonoff) and Hausdorff, since $\{0,1\}$ is compact Hausdorff.
 \begin{rem}
 \emph{Note that in the definition of $|_v$, we do not require $v(y) \neq 0$}. 
 \end{rem}
\begin{lem}\label{Step2}
The map $\phi: Spv(A) \longrightarrow P(A \times A)$, $v \mapsto |_v$ is an injection and a binary relation $|$ belongs to $im(\phi)$ if and only if $|$ satisfies the following axioms for all $a,b,c \in A$:\smallskip
\newline(1) $a|b$ or $b|a$\\
(2) If $a|b$ and $b|c$ then $a|c$\\
(3) If $a|b$ then $ac|bc$ \\ 
(4) If $ac|bc$ and $0 \nmid c$ then $a|b$\\
(5) $0 \nmid 1$
\end{lem} 
\begin{proof}
Let $w$ and $w'$ be two non-equivalent valuations on $A$. Then, there exists $x,y \in A$ such that $w(x) \leq w(y)$ but $w'(x) \nleq w'(y)$ (or conversely) by Proposition \ref{equival} $(iii)$. Thus, $|_w \neq |_{w'}$. This shows that $\phi$ is injective. 
\newline Obviously, the axioms stated in this lemma are satisfied by $|_v$ for any valuation $v$ on $A$. Let us now verify that the above axioms encode exactly the image of $\phi$. Fix such a binary
relation $|$ satisfying all the axioms and declare $a \sim b$ if $a|b$ and $b|a$. Let $[a]$ denote the equivalence class of $a$. Multiplication of equivalence classes can be defined via multiplication of representatives and can be verified to be well defined easily. If $ab \sim 0$ , then $0|ab$. Since $0 = 0a$, if $0\nmid a$, then by axiom (4) we get $0|b$. Hence, $ab \sim 0 $ if and only if either $a \sim 0$ or $b \sim 0$. Thus, multiplication of non zero equivalence classes gives a non zero equivalence class and therefore we get the unpointed monoid
$M :=\{ [a] \  | \ \ [a]\neq 0\}$. It is a commutative monoid with identity ( the class of 1 which is distinct from that of 0 by axiom (5)). By axiom (4), it is also cancellative. \vspace{2mm}
\newline Take distinct $[a]$ and $[b]$ from $M$. Then either $a|b$ or $b|a$ by axiom (1). If $a|b$, then we set $[a] \leq [b]$. This ordering is independent of the choice of representatives of the equivalence classes by axiom (2). Thus, $M$ is a totally ordered cancellative commutative monoid and we can create an abelian group corresponding to this monoid. Declare two points $(m_1,m_2)$ (think of it as $m_1/m_2$) and $(m_1',m_2')$ equivalent if $m_1m_2' = m_1'm_2$ where $m_1,m_2,m_1',m_2' \in M$. Clearly this relation is reflective and symmetric. Transitivity follows by axioms (3) and (4). Therefore, this gives an equivalence relation and we denote each equivalence class $(m_1,m_2)$ as a fraction $m_1/m_2$. We multiply two classes in the obvious way and  it is well-defined due to the cancellative property. If $\Gamma$ denote the set of equivalance classes, then it is clearly an abelian group with identity $1/1$ and $(m_1/m_2)^{-1} = m_2/m_1$.\vspace{2mm}
\newline Since $M$ is cancellative, the canonical morphism $M \longrightarrow \Gamma$ is an injection. We extend the total ordering from $M$ to $\Gamma$ in the usual way, i.e., $m_1/m_2 \leq m_1'/m_2'$ when $m_1m_2' \leq m_1'm_2$. This is well defined by axiom (3) and the cancellative property of $M$. We define,
$v: A \longrightarrow \Gamma \cup \{0\}$ by
\[
    a \mapsto \left\{
                \begin{array}{ll}
                  {[a]/1}, \ \ \textnormal{when \ } a \nsim 0\\
                  {[0]/1}, \ \ \textnormal{when \ } a \sim 0
                \end{array}
              \right.
  \]

It is easy to check that $v$ is a valuation and that $| = |_v$. Hence, the axioms listed
in the lemma entirely encodes the image of $\phi$.
\end{proof} 
\begin{lem}\label{Step3}
The image of $\phi : Spv(A) \longrightarrow P(A \times A),~ v\mapsto |_v$ is a closed compact subset of $P(A \times A)$.
\end{lem}
\begin{proof}
Each of the axioms (for fixed $a,b,c$) as stated in Lemma \ref{Step2} defines a closed subset of $P(A \times A)$. For example, let us see that axiom (2) defines a closed condition. Recall that $P(A \times A)$ can be identified with $\{0,1\}^{A \times A}$. Let $\pi_{a,b}: \{0,1\}^{A \times A} \longrightarrow \{0,1\}$ be  the projection onto the $(a,b)$-th component. Then the collection of elements of $\{0,1\}^{A \times A}$ satisfying axiom (2) is the union of $\pi_{a,c}^{-1}(1)$ and the complement of $\pi_{a,b}^{-1}(1) \cap \pi_{b,c}^{-1}(1)$. Clearly, this collection gives a closed set since $\{0,1\}$ has the discrete topology and $\{0,1\}^{A \times A}$ has the product topology, which makes the projection maps continuous.

Hence the image of $\phi$, being
the set of all binary relations $|$ on $A$ which satisfies the axioms in the Lemma \ref{Step2}, is closed. Since $P(A \times A)$ is compact, the closed set $im(\phi)$ of $P(A \times A)$ is also compact.
\end{proof}

\begin{Thm}\label{Hoch}
Let $X'$ be a quasi-compact topological space and let $\mathcal{U}$ be a collection of clopen sets of $X'$. Endow the set underlying $X'$ with the topology which has $\mathcal{U}$ as an open subbasis and call the resulting topological space $X$. Assume that $X$ is $T_0$. Then $X$ is a spectral space and the basis obtained from the subbasis $\mathcal{U}$ gives a basis of open quasi-compact
subsets of $X$. Moreover, $X_{patch} = X'$. Conversely, every spectral space arises from its patch space in this
way.

\end{Thm}
\begin{proof}
See for instance \cite[Proposition 7]{H} and \cite[Proposition 3.31]{Wd}. 
\end{proof}
We define $U(a,b) := \{v \in Spv(A)\ |\  a|_vb\}= \{v \in Spv(A)\ |\ v(a)\leq v(b)\}$. It follows immediately from the definition of $|_v$ that $b|_v0$ if and only if $v(b)=0$, i.e., $U(b,0)=\{v\in Spv(A)~|~v(b)=0\}$. Thus clearly by definition, $Spv(A)(\frac{a}{b}) = U(a,b) \cap (Spv(A) \setminus U(b,0))$.
\begin{Thm}\label{4.23}
Let $A$ be a monoid.\\
$(i)$ The valuation spectrum $Spv(A)$ is a spectral space with a basis of quasi-compact open subsets given by the subbasis consisting of sets of the form $Spv(A)\left(\frac{a}{b}\right)$ for $a,b \in A$. The Boolean algebra generated by the sets $Spv(A)\left(\frac{a}{b}\right)$ gives a basis for the patch topology on $Spv(A)$. 
\smallskip
\newline $(ii)$ Any monoid morphism $f: A \longrightarrow B$ induces a spectral map $$Spv(f): Spv(B) \longrightarrow Spv(A) \ \ \ \ \  v \mapsto v \circ f$$ We obtain a contravariant functor from the category of monoids to the category of spectral spaces and spectral maps.
\end{Thm}
\begin{proof}
$(i)$  Consider the preimage $\pi^{-1}_{a,b}\{1\}$ for the projection $P(A \times A)\cong \{0,1\}^{A\times A} \xrightarrow{\pi_{a,b}}\{0,1\}$ to the $(a,b)$-th component. It is easy to see that $\pi^{-1}_{a,b}\{1\}$ meets the image of $\phi : Spv(A) \hookrightarrow P(A \times A)$ in exactly $U(a,b)$. Thus, in other words, we have $U(a,b)= \pi^{-1}_{a,b}\{1\}\cap im(\phi)$. The projection $\pi_{a,b}$ is a continuous map to the discrete space $\{0,1\}$. Thus, $U(a,b)$, which is the preimage of the clopen set $\{1\}$, must be clopen in $im(\phi)$. This means $U(a,b)$ is clopen in $im(\phi)$ for any $(a,b) \in A \times A$. Hence, $Spv(A)\left(\frac{a}{b}\right) = U(a,b) \cap (Spv(A)\setminus U(a,0))$ is also clopen in $im(\phi)$. We set,
$ \mathcal{U} :=\{Spv(A)\Big(\frac{a}{b}\Big) \ | \ a,b \in A\}$. Endowing the set underlying $im(\phi)$ with the topology generated by $\mathcal{U}$, we obtain the topological space $Spv(A)$. Also, $Spv(A)$ is $T_0$ by Lemma \ref{Step1}. Hence, our claim follows by Theorem \ref{Hoch}.\smallskip
\newline $(ii)$ Let $f: A \longrightarrow B$ be a monoid morphism which induces the map $Spv(f): Spv(B) \longrightarrow Spv(A)$ defined by $v \mapsto v \circ f$. Since the sets of the form $Spv(A)\left(\frac{a}{b}\right)$ give a basis of open quasi-compact subsets of $Spv(A)$ and $Spv(f)^{-1}(Spv(A)\left(\frac{a}{b}\right)) = Spv(B)\left(\frac{\phi(a)}{\phi(b)}\right)$ is an open quasicompact subset of $Spv(B)$, we have, $Spv(f)$ is a spectral map.
\end{proof}
\subsection{Subspace of continuous valuations as a spectral space}
Let $x$ and $y$ be two points of a topological space $X$. Then, $x$ is called a {\it generization} of $y$ or that $y$ is called a {\it specialization} of $x$ if $y \in \overline{\{x\}}$.
\begin{lem}\label{4.24}
Let $v: A \longrightarrow \Gamma \cup \{0\}$ be a valuation on a monoid $A$. Let $H$ be a subgroup of $\Gamma_v$. We can define a map,
$$v|_H : A \longrightarrow H \cup \{0\}$$
by setting $v|_H(a) := v(a)$ if $v(a) \in H$ and $v|_H(a) := 0$ if $v(a) \notin H$.
If $H$ is convex and $c\Gamma_v \subseteq H$, then $v|_H$ is a valuation. Moreover, $v|_H$ is a specialization of $v$ in $Spv(A)$.
\end{lem}
\begin{proof}
Clearly, $v|_H(0) = 0$ and $v|_H(1) = 1$. We only need to check that $v|_H(ab) = v|_H(a)v|_H(b)$ for any $a,b \in A$. Suppose $v|_H(ab) = 0$ or in other words $v(ab) \notin H$. But since $v(ab) = v(a)v(b)$, at least one of $v(a)$ and $v(b)$ must not belong to $H$. Thus, at least one of $v|_H(a)$ and $v|_H(b)$ must be $0$ which gives $v|_H(ab) = v|_H(a)v|_H(b)$. Now, let $v|_H(ab) = v(ab)$ or in other words $v(ab) \in H$. Then, to show $v|_H(ab) = v|_H(a)v|_H(b)$, we must show that both $v(a),v(b)\in H$. We have the following possible cases:\smallskip
\newline(1) If $v(a) \geq 1$, $v(b) \geq 1$, then $v(a),v(b) \in H$ since $c\Gamma_v \subseteq H$.\smallskip
\newline(2) If $v(a) <1 $, $v(b) < 1$, then $v(a),v(b) \in H$ since $v(ab) \in H$ (by Lemma \ref{convex}).\smallskip  
\newline(3) If $v(a) \geq 1$, $v(b) < 1$, then $v(a) \in H$ since $c\Gamma_v \subseteq H$ and hence $v(b) = v(ab)v(a)^{-1} \in H$.\smallskip
\newline To show that $v|_H$ is a specialization of $v$, it is enough to check that $v|_H(a) \leq v|_H(b)\neq 0$ implies $v(a) \leq v(b) \neq 0$ for $a,b \in A$. Clearly, $v|_H(b)\neq 0$ shows that $v(b) \in H$ and in particular $v(b) \neq 0$. If possible, assume $v(a)>v(b)$. Then, $v(a) \notin H$ and hence $v(a)<1$ as $H$ contains $c\Gamma_v$. Thus, $v(b)<v(a)<1$ and hence $v(a) \in H$ as $H$ is convex. This gives the required contradiction.
\end{proof}
\begin{lem}\label{4.25}
Let $v:A \longrightarrow \Gamma \cup \{0\}$ be a valuation on a monoid $A$. Let $H \subseteq \Gamma_v$ be a convex subgroup such that $c\Gamma_v \subsetneq H$. Then, $\mathfrak{I} := \{x \in A\ |\ v(x) \textnormal{\ is cofinal for}\ H\}$ is an ideal of $A$ and $\sqrt{\mathfrak{I}}= \mathfrak{I}$.
\end{lem}
\begin{proof}
Let $a \in A$ and $x \in \mathfrak{I}$. If $v(a) \leq 1$, then $v(ax) \leq v(x)$ and hence $ax \in \mathfrak{I}$. If $ v(a) >1$, then $v(a) \in c\Gamma_v \subseteq H$. Thus, $v(ax) = v(a)v(x)$ is cofinal for $H$ because $c\Gamma_v \subsetneq H$ (by Lemma \ref{4.2}). It is easy to see that $\sqrt{\mathfrak{I}}= \mathfrak{I}$.
\end{proof}
We call an ideal $I$ of $A$ {\it rad-finite} if there exists a finitely generated ideal $J$ of $A$ such that $\sqrt{I}=\sqrt{J}$.
\begin{lem}\label{4.26}
Let $I$ be a rad-finite ideal of $A$. If $v(I) \cap c\Gamma_v = \emptyset$, then there exists a greatest convex subgroup $H$ of $\Gamma_v$ such that $v(x)$ is cofinal for $H$ for all $x \in I$. Moreover, if $v(I) \neq \{0\}$, then $c\Gamma_v \subseteq H$ and $v(I) \cap H \neq \emptyset$.
\end{lem}
\begin{proof}
If $v(I)=\{0\}$ we can choose $H = \Gamma_v$. So let us assume that $v(I) \neq \{0\}$ and hence $c\Gamma_v \neq \Gamma_v$. Note that for any subgroup $H$ of $\Gamma_v$, we have $v(I) \cap H = \emptyset \iff v(\sqrt{I}) \cap H = \emptyset$. This together with Lemma \ref{4.25}, shows that it is enough to prove the lemma for a finitely generated ideal $I$.  Let $T$ be
a finite set of generators of $I$ and let $H$ be the convex subgroup of $\Gamma_v$ generated by $h:= max\{v(t)\ |\ t\in T\}$. Since $v(I) \cap c\Gamma_v = \emptyset$, we must have, $h <1$ and thus $h$ is cofinal for $H$. Consequently, $v(t)$ is cofinal for $H$ for every $t \in T$.
Moreover, $v(I) \cap c\Gamma_v = \emptyset$ implies $v(x) < c\Gamma_v$ for all $x \in I$. For if $v(x)\geq 1$, then $v(x) \in c\Gamma_v$, giving a contradiction and if $\gamma \leq v(x) \leq 1$ for some $\gamma \in c\Gamma_v$, then $v(x) \in c\Gamma_v$ (as $c\Gamma_v$ is convex), giving a contradiction again. Hence, $v(x) < c\Gamma_v$ for all $x \in I$ and so in particular we have, $h < c\Gamma_v <h^{-1}$. Since $H$ is a convex subgroup generated by $h$, this implies $c\Gamma_v \subsetneq H$. Since $T$ generates $I$ and we have shown that $v(t)$ is cofinal for $H$ for all $t \in T$, it follows from Lemma \ref{4.25} that $v(x)$ is cofinal for $H$ for all $x \in I$.\smallskip
\newline Let $H'$ be another convex subgroup of $\Gamma_v$ such that $v(x)$ is cofinal for $H'$ for all $x \in I$. Then, in particular, $h$ is cofinal for $H'$ which implies that $H' \subseteq H$.
\end{proof}
Lemma \ref{4.26} allows us to define the following subgroup of $\Gamma_v$, for every valuation $v$ of a monoid $A$ and every rad-finite ideal $I$:
let $c\Gamma_v(I)$ be the group $c\Gamma_v$ if $v(I) \cap c\Gamma_v \neq 0$ and otherwise let  $c\Gamma_v(I)$ be the greatest convex subgroup $H$ of $\Gamma_v$ such that $v(x)$ is cofinal for $H$ for all $x \in I$. Note that by Lemma \ref{4.26}, $c\Gamma_v(I)$ is a convex subgroup of $\Gamma_v$ which always contains $c\Gamma_v$. Hence, by Lemma \ref{4.24}, $v|_{c\Gamma_v(I)}$ is a valuation.
\begin{lem}\label{4.29}
Let $v$ be a valuation of $A$ and $I$ be a rad-finite ideal of $A$. Then, the following conditions are equivalent:\\
$(i)~ \Gamma_v = c\Gamma_v(I)$. \\
$(ii)~ \Gamma_v = c\Gamma_v$ or $v(x)$ is cofinal in $\Gamma_v$ for every $x \in I$.\\
$(iii)~ \Gamma_v = c\Gamma_v$ or $v(x)$ is cofinal in $\Gamma_v$ for every element $x$ of a set of generators of $I$.
\end{lem}
\begin{proof}
The equivalence of $(i)$ and $(ii)$ follows from the definition of $c\Gamma_v(I)$ and the equivalence of $(ii)$ and $(iii)$ follows from Lemma \ref{4.25}.
\end{proof}
Let $I$ be a rad-finite ideal of $A$. We set $Spv(A,I): = \{v \in Spv(A)\ |\ \Gamma_v = c\Gamma_v(I) \}$ and equip it with the subspace topology of $Spv(A)$. Note that we get the same subspace if we replace $I$ by any ideal $J$ such that $\sqrt{I} = \sqrt{J}$. We have the following map:
\begin{equation}\label{eq4.2}
r: Spv(A) \longrightarrow Spv(A,I)\ \ \ \ \ v \mapsto v|_{c\Gamma_v(I)}
\end{equation}
Clearly, $r(v) = v$ for every $v \in Spv(A,I)$ i.e., $r$ is a retraction.
\begin{lem}\label{4.30}
Let $\mathfrak{R}$ be the set of all subsets $U$ of $Spv(A,I)$ of the form:
$$U = \{ v \in Spv(A,I)\ |\ v(x_1) \leq v(y)\neq 0,\hdots, v(x_n)\leq v(y)\neq 0\}$$
where $y,x_1, \dots, x_n \in A $ with $I \subseteq \sqrt{\langle x_1,\hdots,x_n\rangle}$.
Then $\mathfrak{R}$ is a basis of the topology of $Spv(A,I)$ that is closed under finite intersections.
\end{lem}  
\begin{proof}
Clearly, each set in $\mathfrak{R}$ is open in $Spv(A,I)$ since it has the induced topology of $Spv(A)$. Let us now show that $\mathfrak{R}$ is closed under finite intersections. Let $x_0,\hdots,x_n$ and $y_0,\hdots, y_m$ be elements of $A$. Then,
$$\Big(\bigcap_{i=1}^n \{ v \in Spv(A)\ |\ v(x_i) \leq v(x_0) \neq 0\}\Big) \bigcap \Big(\bigcap_{j=1}^m \{ v \in Spv(A)\ |\ v(y_i) \leq v(x_0) \neq 0\}\Big) $$ 
$$ = \bigcap_{i=1,\hdots,n \ ; \ j=1,...,m} \{v \in Spv(A)\ |\ v(x_iy_j) \leq v(x_0y_0) \neq 0\}$$
and, if $I \subseteq \sqrt{\langle\{x_i|\ i= 0,...,n\}\rangle}$ and $I \subseteq \sqrt{\langle\{y_j|\ j= 0,...,m\}\rangle}$, then, 
$$I \subseteq \sqrt{\langle\{x_iy_j|\ i= 0,...,n\ ;\ j= 0,...,m\}\rangle}$$ Hence, $\mathfrak{R}$ is closed under finite intersections.\smallskip
\newline Now let us show that $\mathfrak{R}$ gives a basis for the topology on $Spv(A,I)$. Let $v \in Spv(A,I)$ and let $U$ be an open neighborhood of $v$ in $Spv(A)$. We choose $x_0,\hdots,x_n \in A$ such that $$v \in W:= \{w \in Spv(A)\ |\ w(x_i) \leq w(x_0) \neq 0 \textnormal{ \ for all \ } i=1,\hdots ,n\} \subseteq U$$
We can make this choice because the sets like $W$ form a basis for the topology of $Spv(A)$. Now we have two possible cases, $\Gamma_v = c\Gamma_v$ or $\Gamma_v \neq c\Gamma_v$. First assume $\Gamma_v = c\Gamma_v$. Since $v(x_0) \in c\Gamma_v$, there exists some $d \in A$ such that $v(x_0d) \geq 1$. Hence, 
$$v \in W':= \{w \in Spv(A)\ |\ w(x_id) \leq w(x_0d) \neq 0 \textnormal{ \ for all \ } i=1,\hdots ,n \textnormal{ and }w(1) \leq w(x_0d) \neq 0 \} \subseteq W$$ Thus, we have, $W' \cap Spv(A,I) \in \mathfrak{R}$. 
Next assume $\Gamma_v \neq c\Gamma_v$. Let $\{s_1,s_2, \hdots ,s_m\}$ be a finite set of generators of a finitely generated ideal $J$ for which $\sqrt{I}=\sqrt{J}$. By Lemma \ref{4.29}, $v(s_j)$ is cofinal in $\Gamma_v$ for all $j=1,\hdots,m$. Thus we can choose a $k\in \mathbb{N}$ such that $v(s_j^k) \leq v(x_0)$ for all $j=1,\hdots,m$. Hence,
$$v \in W':= \{w \in Spv(A)\ |\ w(x_i) \leq w(x_0) \neq 0  \textnormal{ and  } w(s_j^k) \leq w(x_0) \neq 0 \textnormal{\ for all\ }i,j \} \subseteq W$$
where $i=1,\hdots,n$ and $j=1,\hdots,m$. Thus, we have $W' \cap Spv(A,I) \in \mathfrak{R}$. This completes the proof.
\end{proof}
\begin{lem}\label{4.31}
Let $I$ be an ideal of a monoid $A$ and $y,x_1, \hdots, x_n \in A$ such that $I \subseteq \sqrt{\langle x_1,\hdots, x_n\rangle}$. For
$$U = \{ v \in Spv(A,I)\ |\ v(x_1) \leq v(y)\neq 0,\hdots, v(x_n)\leq v(y)\neq 0\}$$
$$W = \{ v \in Spv(A)\ |\ v(x_1) \leq v(y)\neq 0,\hdots, v(x_n)\leq v(y)\neq 0\}$$
we have, $r^{-1}(U) = W$ where $r$ is as in \eqref{eq4.2}. In particular, for any $v \in Spv(A)$ if $v(I) \neq 0$, then $r(v)(I) \neq 0$.
\end{lem}
\begin{proof}
Since $U \subseteq W$ and as every point $v \in r^{-1}(U)$ specializes to a point $r(v)$ of $U$, we have $r^{-1}(U) \subseteq W$ because $W$ is open in $Spv(A)$. \smallskip
\newline Conversely, take any $w \in W$. We want to show that $r(w) \in U$. If $w(I) =0$, then $c\Gamma_w(I) = \Gamma_w$ by definition and thus $r \in Spv(A,I)$. Hence, $r(w) = w$. Now assume $w(I) \neq 0$. We claim that $r(w)(x_i) \leq r(w)(y)$ for all $i=1,\hdots,n$. The inequality is obvious if $r(w)(x_i)=0$ i.e., when $w(x_i) \notin c\Gamma_w(I)$. If $w(x_i) \in c\Gamma_v(I)$, then $w(y)$ also belongs to $c\Gamma_v(I)$. This is because either $w(y) \geq 1$ in which case $w(y) \in c\Gamma_v(I)$ since $c\Gamma_v \subseteq c\Gamma_v(I)$, or $w(x_i)\leq w(y) \leq 1$ in which case $w(y) \in c\Gamma_v(I)$ since $c\Gamma_v(I)$ is convex. Thus, we have $r(w)(x_i) = w(x_i)\leq w(y) = r(w)(y)$. It
remains to show that $r(w)(y) \neq 0$. So, if possible, let $r(w)(y) = 0$. Then $r(w)(x_i) = 0$ for all $i=1,\hdots,n$. Since $I \subseteq \sqrt{\langle x_1,\hdots, x_n\rangle}$, this implies $r(w)(x) = 0$ for all $x\in I$. But by Lemma \ref{4.26}, $w(I) \neq 0 $ implies $w(I) \cap c\Gamma_w(I) \neq \emptyset$, i.e., there
exists some $x \in I$ such that $r(w)(x) \neq 0$. This gives the required contradiction.\smallskip
\newline Note that we have also proved that for any $v \in Spv(A)$ if $v(I) \neq 0$, then $r(v)(I) \neq 0$.
\end{proof}

\begin{Thm}\label{4.32}
The topological space $Spv(A,I)$ is spectral and the set $\mathfrak{R}$ (as in Lemma \ref{4.30}) forms a basis of quasi-compact open subsets of the topology, which is stable under
finite intersections. Moreover, the retraction $r: Spv(A) \longrightarrow Spv(A,I)$ (as in \eqref{eq4.2}) is a spectral map.
\end{Thm}

\begin{proof}
Let $\overline{\mathfrak{R}}$ denote the Boolean algebra of subsets of $Spv(A,I)$ generated by $\mathfrak{R}$. Let $X$ denote the set underlying $Spv(A,I)$. Endow $X$ with the topology generated by $\overline{\mathfrak{R}}$. By Lemma \ref{4.31} and Theorem \ref{4.23}(i), we see that $r: (Spv(A))_{patch} \longrightarrow X$ is continuous. Since $(Spv(A))_{patch}$ is quasi-compact (see, for instance,\cite[Theorem 1]{H}) and $r$ is surjective, $X$ is quasi-compact. Since $X$ has the topology generated by $\overline{\mathfrak{R}}$, every set of $\mathfrak{R}$ is clopen in $X$. By Lemma \ref{4.30}, endowing $X$ with the topology generated by $\mathfrak{R}$ gives the topological space $Spv(A,I)$. Since $Spv(A)$ is $T_0$, so is $Spv(A,I)$. Hence by Theorem \ref{Hoch}, $Spv(A,I)$ is a spectral space and $\mathfrak{R}$ is a basis of open quasi-compact subsets of $Spv(A,I)$. By Lemma \ref{4.30}, $\mathfrak{R}$ is stable under
finite intersections and by Lemma \ref{4.31}, $r: Spv(A) \longrightarrow Spv(A,I)$ is a spectral map.
\end{proof}
A {\it topological monoid} is a monoid $A$ endowed with a topology such that the monoid operation
is continuous. Monoids endowed with the discrete and indiscrete topologies or the underlying multiplicative monoid of a topological ring are some simple examples of topological monoids. Some interesting examples are:
\begin{ex}
\emph{(1) $MSpec(A)$ is a topological monoid for any commutative monoid $A$ (see, \cite[$\S$ 2]{IP}).}\smallskip
\newline\emph{(2) Given any topological space $X$, we can construct a topological monoid $A := X \bigcup \{1,0\}$ whose underlying space is $X$ adjoined with two disjoint points 1 and 0. The point 1 represents the identity of the monoid and 0 represents the basepoint. The monoid operation is defined as : $a.b = 0$ for all $a,b \neq 1$ and $a.1=1.a =a $. We declare a set $U \subseteq A$ is open if and only if $U \cap X$ is open. It is easy to check that this gives a topology on $A$ and also makes it a topological monoid.}\smallskip
\newline\emph{(3) Given a topological group $(G,*)$ we can obtain a monoid $A:= G \cup \{0\}$ by adjoining a basepoint $0$ to $G$. The monoid operation is defined as : $a.b = a*b$ for all $a,b \neq 0$ and $a.0=0.a =0$. We declare a set $U \subseteq A$ is open if and only if $U \cap G$ is open. This makes $A$ a topological monoid.}\smallskip
\end{ex}
We now introduce the notion of $I$-topology on a monoid induced by an ideal $I$. This is inspired by the classical $I$-adic topology on a ring induced by an ideal $I$ of the ring (see, for instance, \cite[Chapter 10]{AM}).
\begin{defn}\label{I-top}
 Given an ideal $I$ of a monoid $A$, we can define a topology on the monoid by declaring
$U \subseteq A$ is open if either $0 \notin U$ or $I^n \subseteq U$ for some $n \in \mathbb{N}\cup\{0\}$ where $I^{0}:=A$. It can be easily checked that the monoid multiplication is continuous with respect to this topology and hence this gives a topological monoid. We call such a topology on a monoid defined by an ideal $I$ as {\it $I$-topology}. \end{defn}
If an ideal $I$ of a monoid $A$ is such that $\bigcap\limits_{n=0}^{\infty}I^n = \{0\}$, then it is easy to check that the $I$-topology is the same as the topology induced by the following metric:
\[
    d(a,b) = \left\{
                \begin{array}{ll}
                  \frac{1}{2^n}, \ \ \textnormal{for\ } a \neq b \textnormal{ and for the largest non-negative integer\ } n \textnormal{\ such that both\ } a,b \in I^n\\
                  \ 0, \ \ \textnormal{\ for \ } a=b
                \end{array}
              \right.
  \]
It may be easily checked that this gives a complete metric space!
\begin{defn}\label{Contdef}
Let $v$ be a valuation on a topological monoid $A$ with value group $\Gamma_v$. We call $v$ {\it continuous} if the map $v: A \longrightarrow \Gamma_v \cup \{0\}$ is continuous where $\Gamma_v \cup \{0\}$ is endowed with the following topology: $U \subseteq \Gamma_v \cup \{0\}$ is open if $0 \notin U$ or if there exists $\gamma \in \Gamma_v$ such that the set $\{v(x)\in \Gamma_v\ |\  v(x) < \gamma\} \subseteq U$. We denote by $Cont(A)$ the subspace of $Spv(A)$ consisting of the continuous valuations on $A$.
\end{defn}
We now discuss examples of continuous and non-continuous valuations.
\begin{ex}
\emph{For any abelian group $(G,*)$, consider the cancellative monoid $A:=G\cup \{0\}$ obtained by adjoining a basepoint $0$ to $G$. Explicitly, the monoid operation is defined as : $a.b = a*b$ for all $a,b \neq 0$ and $a.0=0.a =0$. Take any non-trivial ideal $I$ of $A$, say, for instance, $I=\langle g \rangle$ where $g$ is a non-identity element of the group $G$. Endow $A$ with the $I$-topology. Now, consider the following valuation $v:A\longrightarrow \{0,1\}$ of $A$ with the value group $\Gamma_v=\{1\}$: 
\[
    v(a) = \left\{
                \begin{array}{ll}
                  1, \ \ \text{when }a \neq 0\\
                  0, \ \ \text{when }a = 0
                \end{array}
              \right.
  \]
By Definition \ref{Contdef}, the topology on $\{0,1\}$ is discrete. But, $v^{-1}(\{0\})=\{0\}$ is not open in the $I$-topology of $A$. This is because $I^n\neq 0$ for all $n \in \mathbb{N}$, as $A$ is cancellative. Thus, this gives an example of non-continuous valuation. If we take $I=\{0\}$, then the $I$-topology on $A$ will make the valuation $v$ continuous. In fact, any valuation on a monoid having discrete topology is continuous.
}
\end{ex}
Let $A$ be a topological monoid. We call an element $x \in A$ {\it topologically nilpotent} if
$x^n \longrightarrow 0$ as $n \longrightarrow \infty$.
Note that for a topological monoid $A$ having the $I$-topology for some ideal $I$, $\sqrt{I}$ gives the complete collection of topologically nilpotent elements of $A$.\smallskip
\begin{Thm}\label{Contval}
Let $A$ be a monoid and $I$ be a finitely generated ideal of $A$. Endowing $A$ with the $I$-topology we obtain a topological monoid. Then,\smallskip
\newline$(i)$ For a valuation $v:A \longrightarrow \Gamma_v\cup\{0\}$, $$v \in Cont(A)\Longleftrightarrow \{x \in A\ |\ v(x)< \gamma\} \textnormal{ is open for all } \gamma \in \Gamma_v$$
$(ii)~Cont(A) = \{v \in Spv(A,I)\ |\ v(x) <1 \textnormal{ for all } x \in I\}.$\\
$(iii)~Cont(A)$ is a spectral space. 
\end{Thm}
\begin{proof}
$(i)$ If $v \in Cont(A)$, then clearly $\{x \in A\ |\ v(x)< \gamma\}$ is open for all $\gamma \in \Gamma_v$. To show the other way, consider any open set $U$ of $\Gamma_v \cup \{0\}$. If $U$ is such that $0 \notin U$, then $0 \notin v^{-1}(U)$ and hence $v^{-1}(U)$ is open in $A$. If $U$ is such that there exists some $\gamma \in \Gamma_v$ for which $\{v(x)\in \Gamma_v\ |\  v(x) < \gamma\} \subseteq U$, then the open set $ \{x \in A\ |\ v(x)< \gamma\}$ is contained in $v^{-1}(U)$. Since $ \{x \in A\ |\ v(x)< \gamma\}$ is an open set of $A$ containing 0, there exists some $n \in \mathbb{N}$ such that $I^n \subseteq  \{x \in A\ |\ v(x)< \gamma\} \subseteq  v^{-1}(U)$. Hence, $v^{-1}(U)$ is open in $A$.\smallskip
\newline $(ii)$ Let $v \in Cont(A)$. Then, as we have just shown, $\{x \in A\ |\ v(x)< \gamma\}$ is an open neighborhood of 0 for all $\gamma \in \Gamma_v$. Since every element $a \in I$ is topologically nilpotent, for every $\gamma \in \Gamma$ there exists some $n \in \mathbb{N}$ such that $a^n \in \{x \in A\ |\ v(x)< \gamma\}$ i.e., $v(a)^n < \gamma$. Thus, $v(a)$ is cofinal in $\Gamma_v$ for every $a \in I$. This implies that $v(a) <1$ as $1 \in \Gamma_v$. By Lemma \ref{4.29}, we also have $\Gamma_v = c\Gamma_v(I)$. Hence, $v \in   \{v \in Spv(A,I)\ | \ v(x) <1 \textnormal{ \ for all \ } x \in I\}$.\\
Conversely, let $v \in Spv(A,I)$ with $v(a) <1$ for all $a \in I$. We want to show that $v(a)$ is cofinal in $\Gamma_v$ for all $a \in I$. If $c\Gamma_v \neq \Gamma_v$, we are done by Lemma \ref{4.29}. Now let $c\Gamma_v = \Gamma_v$ and let $\gamma \in \Gamma_v$ be any given element. For $\gamma \geq 1$, we are done by our hypothesis that $v(a) <1$. Otherwise, if $\gamma <1$, then by the definition of the characteristic subgroup there exists $t\in A$ such that $v(t) \neq 0$ and 
$v(t)^{-1}\leq \gamma <1 $. Note that the map 
$A \longrightarrow A $ defined by $x \mapsto tx$ is continuous since $A$ is a topological monoid. Thus the preimage of the open set $I$ under this map is open and is given by
$\{x \in A\ |\  tx \in I\}$. Since this is an open neighborhood of 0 in $A$, there exists some $n \in \mathbb{N}$ such that $I^n \subseteq \{x \in A\ | \ tx \in I\}$. For any $a \in I$, we therefore have $ta^n \in I$. Hence, by assumption $v(ta^n) < 1$ and so $v(a)^n < \gamma$. This proves our claim. \smallskip
\newline Let $\mathcal{S} =\{a_1,a_2, \hdots a_m\}$ be a finite set of generators of $I$. Set $\delta := max\{v(a_i)\}$. Then, there exists $n \in \mathbb{N}$ such that $\delta^n < \gamma$. Now, let $x$ be any element in $I^{n+1}$. Then, $x = aa_{x_1}\hdots a_{x_{n+1}}$ for some $a\in A$ and $a_{x_i} \in \mathcal{S}$ for $i=1,\hdots,n+1$. Hence, $v(x)=v(aa_{x_1}\hdots a_{x_{n+1}}) = v(aa_{x_1})v(a_{x_2})\hdots v(a_{x_{n+1}}) < v(aa_{x_1})\delta^n < \gamma$ since $aa_{x_1} \in I$. So by assumption $v(aa_{x_1})<1$. Thus, $I^{n+1} \subseteq \{x\in A\ |\ v(x)< \gamma\}$. This shows $\{x\in A\ | \ v(x)< \gamma\}$ is open for any $\gamma \in \Gamma_v$ and hence by $(i)$, $v$ is continuous. This completes the proof of $(ii)$.\smallskip
\newline $(iii)$ Note that $Spv(A,I)\setminus Cont(A) = \bigcup_{x \in I} \{v \in Spv(A,I)\ |\ 1\leq v(x)\}$ is open in $Spv(A,I)$ since each set on the right-hand side is open by Lemma \ref{4.30}. Thus, $Cont(A)$ is a closed subspace of $Spv(A,I)$. Recall that $Spv(A,I)$ is a spectral space by Theorem \ref{4.32}. Since any closed subspace of a spectral space is spectral \cite[Remark 3.13]{Wd}, we have that $Cont(A)$ is spectral.
\end{proof}

\end{document}